\documentclass[12pt]{amsart}
\usepackage[english]{babel}
\usepackage{amsmath,amsthm,amsfonts,amssymb,epsfig,color,bbm, color,enumerate,caption,subcaption,hyperref}
\usepackage{vector} 
\usepackage[left=1in,top=1in,right=1in]{geometry}
\usepackage{enumitem}

\newtheorem{thm}{Theorem}[section]
\newtheorem{lem}[thm]{Lemma}
\newtheorem{prop}[thm]{Proposition}
\newtheorem{cor}[thm]{Corollary}

\newtheorem{rem}{Remark}
\newtheorem*{conj}{Conjecture}

\newcommand{\E}{\mathbb{E}}

\newcommand{\prob}{\mathbb{P}}
\newcommand{\R}{\mathbb{R}}

\newcommand{\N}{\mathbb{N}}

\newcommand{\F}{\mathcal{F}}

\newcommand{\e}{\varepsilon}
\newcommand{\re}{\text{Re}}

\title[Maxima of a randomized Riemann zeta function]{Maxima of a randomized Riemann zeta function, \\and branching random walks}
\thanks{The authors would like to thank the Isaac Newton Institute for Mathematical Sciences and the Centre de Recherches Math\'ematiques for hospitality and financial support during part of this work. }

\author[L.-P. Arguin]{Louis-Pierre ARGUIN}
\thanks{L.-P. A. is supported by a NSERC discovery grant and a grant FQRNT {\it Nouveaux chercheurs}.}
\address{L.-P. Arguin\\D\'epartement de Math\'ematiques et Statistique, Universit\'e de Montr\'eal, Montr\'eal QC H3T 1J4, Canada; 
Department of Mathematics, Baruch College and Graduate Center, City University of New York, New York, NY 10010, USA.
}
\email{arguinlp@dms.umontreal.ca}

\author[D. Belius]{David Belius}
\thanks{D. B. was supported by the Swiss National Science Foundation, the Centre de Recherches Math\'{e}matiques and the Institut des Sciences Math\'{e}matiques, for part of the research leading up to this work.}
\address{D. Belius\\ Courant Institute of Mathematical Sciences, New York University, New York, NY 10012, USA.}
\email{david.belius@cantab.net}

\author[A. J. Harper]{Adam J. Harper}
\thanks{A. J. H. is supported by a research fellowship at Jesus College, Cambridge.}
\address{A. J. Harper\\ Jesus College, Cambridge, CB5 8BL, England.}
\email{A.J.Harper@dpmms.cam.ac.uk}

\keywords{Extreme Value Theory, Riemann Zeta function, Branching Random Walk} 
\subjclass[2000]{60G70, 11M06}

\begin{document}

\begin{abstract}
A recent conjecture of Fyodorov--Hiary--Keating states that the maximum of the absolute value of the Riemann zeta function on a typical bounded interval of the critical line is
$\exp\{\log \log T -\frac{3}{4}\log \log \log T+O(1)\}$, for an interval at (large) height $T$. 
In this paper, we verify the first two terms in the exponential
for a model of the zeta function, which is essentially a randomized Euler product.
The critical element of the proof is the identification of an approximate tree structure,
present also in the actual zeta function, 
which allows us to relate the maximum to that of a branching random walk.
\end{abstract}

\maketitle

\section{Introduction}

The Riemann zeta function is defined for $\re(s) > 1$ by a sum over integers, or equivalently by an {\it Euler product} over primes, as
\begin{equation}
\zeta(s)=\sum_{n=1}^\infty \frac{1}{n^s}=\prod_{p \text{ primes}} (1-p^{-s})^{-1}\ ,
\end{equation}
and by analytic continuation for other complex $s$. The behaviour of the function
on the critical line $\re(s)=1/2$ is a major theme in number theory,
the most important questions of course concerning the zeroes (e.g. the Riemann Hypothesis).

This paper is motivated by the study of the large values of 
$|\zeta(s)|$ on the critical line $s=1/2+it$.
Little is known about the behavior on long intervals, say $0\leq t \leq T$ for $T$ large.
The Lindel\"of hypothesis, which is implied by the Riemann hypothesis, states that $\max_{0\leq t\leq T}|\zeta(1/2+it)|$ grows slower than any small power of $T$. See the paper of Farmer, Gonek and Hughes~\cite{farmer-gonek-hughes} for more precise conjectures about this maximum size, and the paper of Soundararajan~\cite{soundextremes} for a rigorous lower bound. More recently, Fyodorov, Hiary and Keating considered the maximum on bounded intervals of the critical line.
They made the following conjecture:
\begin{conj}[Fyodorov--Hiary--Keating \cite{fyodorov-hiary-keating, fyodorov-keating}]
\label{conj: FHK}
For $\tau$ sampled uniformly from $[0,T]$,
\begin{equation}
\label{eqn: FHK}
\max_{h\in [0,1]} \log |\zeta(1/2+i(\tau +h))|= \log\log T -\frac{3}{4}\log\log\log T + O_P(1)\ ,
\end{equation}
where $O_P(1)$ is a term that is stochastically bounded as $T\to \infty$.
\end{conj}

The main result of this paper is a proof of the validity of the first two terms in \eqref{eqn: FHK} for a {\it random model} of $\zeta$ defined in \eqref{eqn: process} below, which is essentially a randomized Euler product. Until now such precise estimates were not known rigorously even for models of zeta.

The conjecture is intriguing for many reasons. From a number theory point of view, the precision of the prediction is striking. 
From a probability point of view, the leading and subleading order of the maximum correspond exactly to those of the maximum of a branching random walk
(which is a collection of correlated random walks indexed by the leaves of a tree), as will be explained below.
In fact, the key element of the proof for the random model will be the identification of an approximate tree structure for the zeta function.

\subsection{Modelling the zeta function}
If we take logarithms and Taylor expand the Euler product formula for the zeta function, we find for $\re(s)>1$,
\begin{equation}
\label{eqn: Euler product taylor}
\log\zeta(s) = - \sum_{p} \log(1-p^{-s}) = \sum_{k=1}^{\infty} \frac{1}{k} \sum_{p} \frac{1}{p^{ks}} = \sum_{p} \frac{1}{p^{s}} + O(1) , \;\;\;\;\; 
\end{equation}
since the total contribution from all proper prime powers ($p^{ks}$ with $k \geq 2$) is uniformly bounded. 
One of the great challenges of analytic number theory is to understand how the influence of the Euler product may persist for general $s \in \mathbb{C}$. 
The definition of our random model is based on a rigorous result in that direction, assuming the truth of the Riemann Hypothesis, proved by Harper~\cite{harper} by adapting a method of Soundararajan \cite{sound} (which itself builds heavily on classical work of Selberg~\cite{selberg}).
\begin{prop}[See Proposition 1 of Harper \cite{harper}]
\label{prop: rigorous}
Assume the Riemann Hypothesis. For $T$ large enough  there exists a set $\mathcal{H} \subseteq [T,T+1]$, of measure at least $0.99$, such that
\begin{equation}
\label{eqn: sum rigorous}
\log|\zeta(1/2+it)| = \textup{Re}\left(\sum_{p \leq T} \frac{1}{p^{1/2+it}} \frac{\log(T/p)}{\log T}\right) + O(1) \;\;\;\;\; \forall t \in \mathcal{H} .
\end{equation}
\end{prop}

The set $\mathcal{H}$ produced in Proposition \ref{prop: rigorous} consists of values $t$ that are not abnormally close, in a certain averaged sense, to many zeros of the zeta function. It seems reasonable to think that one shouldn't typically find maxima very close to zeros. Moreover, if one only wants an upper bound then the restriction to the set $\mathcal{H}$ can in fact be removed, at the cost of a slightly more complicated
right-hand side. Therefore, to understand the typical size of $\max_{0 \leq h \leq 1} \log|\zeta(1/2+i(\tau+h))|$ as $\tau$ varies we should try to understand the typical size of $\max_{0 \leq h \leq 1} \sum_{p \leq T} \re\left(\frac{1}{p^{1/2+i(\tau + h)}} \frac{\log(T/p)}{\log T}\right)$. The factor $\log(T/p)/\log T$ is a smoothing introduced for technical reasons. For simplicity we shall ignore it in our model.

Since the values of $\log p$ are linearly independent for distinct primes, it is easy to check by computing moments that the finite-dimensional distributions of the process $(p^{-i\tau} , \text{ $p$ primes})$, where $\tau$ is sampled uniformly from $[0,T]$, converge as $T \rightarrow \infty$ to those of a sequence of independent random variables distributed uniformly on the unit circle.
Following \cite{harper}, this observation suggests to build a model from a probability space $(\Omega,\F,\prob)$ with random variables $(U_p,p \text{ primes})$ which are uniform on the
unit circle, and independent. 
For $T>0$ and $h\in \R$, we consider the random variables $\sum_{p\leq T}p^{-1/2}\re (U_p p^{-ih})$.
In view of Proposition \ref{prop: rigorous}, the process
 \begin{equation}
 \label{eqn: process}
\left(\sum_{p\leq T}\frac{\re (U_p p^{-ih})}{p^{1/2}}, h\in [0,1]\right)
\end{equation}
seems like a reasonable model for the large values of $(\log |\zeta(1/2+i(\tau+h))|, h\in [0,1])$.

\subsection{Main Result}
In this paper, we provide evidence in favor of Conjecture \ref{conj: FHK} by proving a similar statement for the random model \eqref{eqn: process}.
At the same time, we hope to outline a possible approach to tackle the conjecture for the Riemann zeta function itself.

\begin{thm}
\label{thm: main}
Let $(U_p,p \text{ primes})$ be independent random variables on $(\Omega,\F,\prob)$, distributed uniformly on the unit circle.
Then
\begin{equation}
\label{eq: main}
\max_{h\in [0,1]} \sum_{p\leq T}\frac{\re (U_p p^{-ih})}{p^{1/2}} =\log\log T -\frac{3}{4}\log\log\log T + o_P(\log\log\log T) , 
\end{equation}
where the sum is over the primes less than or equal to $T$ and the error term converges to $0$ in probability when divided by $\log\log \log T$. 
\end{thm}
An outline of the proof of the theorem is given in Section \ref{sect: outline} below.
The technical tools needed are developed in Section \ref{sect: tools}, and finally the proof
is given in Section \ref{sect: proof}.

\subsection{Relations to previous results}
The leading order term $\log\log T$ in \eqref{eq: main} was proved in \cite{harper}, where it was also shown that the second order term must lie between $-2\log\log\log T$ and $-(1/4)\log\log\log T$. 
As well as giving a stronger result, our analysis here is ultimately based on a control of the joint distribution of only two points $h_1$ and $h_2$ of the random process at a time, which could feasibly be achieved for the zeta function itself. In contrast, the lower bound analysis in \cite{harper} depends on a 
Gaussian comparison inequality that requires control of $\log T$ points.

Fyodorov, Hiary and Keating motivated Conjecture \ref{conj: FHK} in \cite{fyodorov-hiary-keating, fyodorov-keating} using a connection to random matrices.
There is convincing evidence, see e.g.~\cite{keating-snaith}, that the values of the zeta function in an interval of the critical line are well modelled by the characteristic polynomial $P_N(x)$ of an $N\times N$ matrix sampled uniformly from the unitary group,
for $x=e^{i\theta}$ on the unit circle.
In this spirit, they compute in \cite{fyodorov-keating} the moments of the partition function $Z_N(\beta)=\int_{0}^{2\pi}|P_N(e^{i\theta})|^\beta d\theta$.
They argue that these coincide with those previously obtained for a logarithmically correlated Gaussian field \cite{fyodorov-bouchaud}. 
For large $\beta$, this leads to the conjecture that the maximum of the characteristic polynomial behaves like the maximum of the Gaussian model.
Unfortunately, the analogue of Conjecture \ref{conj: FHK} for this random matrix model is not known rigorously even to leading order
(see \cite{webb} for recent developments at low $\beta$ and its relation to Gaussian multiplicative chaos).
The conjecture is also expected to hold for other random matrix models such as the Gaussian Unitary Ensembles, see \cite{fyodorov-simms}.
One advantage of the model \eqref{eqn: process} is that it can be analysed rigorously to a high level of precision with current probabilistic techniques.

As explained in Section \ref{sect: outline}, the proof of Theorem \ref{thm: main} uses in a crucial way an {\it approximate tree structure} present in our model and also in the actual zeta function.
This structure explains the observed agreement between the high values of the zeta function and those of log-correlated random fields. 
The approach to control subleading orders of  log-correlated Gaussian fields and branching random walks was first developed by Bramson \cite{bramson} in his seminal work on the maximum of branching Brownian motion. 
It has since been extended to more general branching random walks by several authors, e.g.~\cite{addario-berry-reed-minima, aidekon, bramson-ding-zeitouni2}, and to log-correlated Gaussian fields, see for example \cite{bramson-ding-zeitouni, madaule}.
This type of argument can also be applied to obtain the joint distribution of the near-maxima, see e.g.~\cite{aidekon-berestycki-brunet-shi, arguin-bovier-kistler, biskup-louidor}.
Recently, a multiscale refinement of the second moment method was introduced by Kistler in \cite{kistler} to control the leading and subleading orders of processes with neither {\it a priori} Gaussianity nor exact tree structure.
It was successfully implemented in \cite{belius-kistler} to obtain the subleading order of cover times on the two-dimensional torus.
The proof of Theorem \ref{thm: main} follows the same approach.

It is instructive to consider the conjecture in the light of the statistics of typical values of the zeta function.
One beautiful result is the {\it Selberg central limit theorem}~\cite{selberg}, which asserts that if $\tau$ is sampled uniformly from the interval $[0,T]$ then
$(\frac{1}{2}\log\log T)^{-1/2}\log |\zeta(1/2+i\tau)|$ converges in law to a standard Gaussian variable.
Thus, to obtain a rough prediction for the order of the maximum on $[0,1]$, one may
compare it to the maximum of independent Gaussian variables of mean $0$ and variance $\frac{1}{2}\log\log T$.
For $\log T$ such variables, it is not hard to show that the order of the maximum is $\log \log T - \frac{1}{4} \log \log \log T + O(1)$.
The leading order agrees with Conjecture \ref{conj: FHK}, but the constant in the subleading correction is different.
Our proof shows how to modify this ``independent'' heuristic to account for the ``extra'' $-\frac{1}{2}\log \log \log T$
present in Conjecture \ref{conj: FHK}.
Bourgade showed a multivariate version of Selberg's theorem  where the correlations are logarithmic in the limit \cite{bourgade}.
However, the convergence is too weak to describe the maximum on an interval.

\subsection{Outline of the proof}
\label{sect: outline}

The proof of Theorem \ref{thm: main} is based on an analogy between the process \eqref{eqn: process} and a branching random walk (also known as hierarchical random field). 
We make this connection precise here, and indicate for an unfamiliar reader how to analyse the maximum of a branching random walk.

We will work in the case where $T=e^{2^n}$ for some large natural number $n$.
In this setup, the process of interest in Theorem \ref{thm: main} is
\begin{equation}\label{eq: X}
(X_n(h), h\in [0,1])\ ,\qquad \text{ where }\  X_n(h)=\sum_{p\leq e^{2^n}}\frac{\re (U_p p^{-ih})}{p^{1/2}}\ 
\end{equation}
is a continuous function of $h$.
Since $\log \log T = n \log 2$ and $\log \log \log T = \log n + O(1)$,
Theorem \ref{thm: main} can be restated as:
\begin{eqnarray}
&& \displaystyle{\lim_{n \to \infty}} \prob\left[ m_n(-\varepsilon) \le \max_{h\in [0,1]} X_n(h) \le m_n(\varepsilon) \right] = 1 \mbox{, for all }\varepsilon>0,\label{eq: main in terms of X}\\
&&   \text{where } m_n(\varepsilon)=n\log 2 -\frac{3}{4}\log n +\e \log n\ .\label{eqn: m_n definition}
\end{eqnarray}
In other words, with large probability, the maximum of the process lies in an arbitrarily small window (of order $\log n$) around $n\log 2 -\frac{3}{4}\log n$.
  
By symmetry of $U_p$ we have $\E[X_{n}(h)] = 0$
for any $h \in [0,1]$.
Also a simple computation shows that $\E [ \re (U_p p^{-ih})\re (U_p p^{-ih'})] = (1/2)\cos(|h-h'| \log p)$,
so the covariance $\E[X_n(h)X_n(h')]$ equals $\frac{1}{2}\sum_{\log p\leq 2^n} p^{-1}\cos( |h-h'| \log p)$.
Using well known results on primes (cf.~Lemma \ref{lem: primes}), it is possible to estimate this as
\begin{equation}
\label{eq: covar} 
  \E  \left[ X_n(h)X_n(h') \right] \approx \frac{1}{2}\log |h-h'|^{-1},
\end{equation}
for any $h,h' \in [0,1]$ provided $|h-h'| \geq 2^{-n}$. 
If instead $|h-h'| < 2^{-n}$, then the covariance is almost $n(\log 2)/2$, i.e. $X_n(h)$ and $X_n(h')$ are almost perfectly correlated.
Therefore, one can think of the maximum over $h \in [0,1]$ as a maximum over $2^{n}$ equally spaced points.

The key point of the proof is that the logarithmic nature of the correlations can be understood in a more structural way using a {\it multiscale decomposition}.
Precisely, we rewrite the process as
\begin{equation}
\label{eqn: Y}
X_n(h)=\sum_{k=0}^n Y_k(h),  \;\;\; \text{where} \; Y_k(h)=\sum_{2^{k-1}< \log p \leq 2^k} \frac{\re (U_p p^{-ih})}{p^{1/2}}  \ ,
\end{equation}
is the {\it increment at ``scale'' $k$} of $X_n(h)$.
It is not hard to show, see Section \ref{sect: ld}, that for $k$ large,
\begin{equation}
\label{roughcorr}
\begin{aligned}
 \E  \left[ Y_k(h)^{2} \right] &
 \approx \frac{\log 2}{2}, \quad\mbox{ and}\quad
\E  \left[ Y_k(h)Y_k(h') \right] &
\approx 
\begin{cases}
\frac{\log 2}{2}  & \text{ if $|h-h'| \leq 2^{-k}$},  \\
0 & \text{ if $|h-h'| > 2^{-k}$.}
\end{cases}
\end{aligned}
\end{equation}
In view of \eqref{roughcorr}, for given $h,h'$, one can think of the partial sums $X_k(h) =\sum_{j=1}^k Y_j(h)$ and $X_k(h') =\sum_{j=1}^k Y_j(h')$ as random walks,
where the increments $Y_j(h), Y_j(h')$ are almost perfectly correlated (so roughly the same) for those $j$
such that $2^{j} \leq |h - h'|^{-1}$, and where they are almost perfectly decorrelated (so essentially independent) when $2^{j} > |h - h'|^{-1}$.
A similar, but exact, behaviour would be obtained as follows:
Consider $2^n$ equally spaced
points in $[0,1]$, thought of as leaves of a binary tree of depth $n$.
Place on each edge of the binary tree an independent Gaussian with mean zero and variance $(\log 2)/2$, and associate to a leaf the random walk given by the partial sums of the Gaussians on the path from root to leaf, see Figure \ref{fig: tree}.
With this construction, the first $k$ increments of the random walks of two leaves will be exactly the same, where $k$ is the level of the most recent common ancestor, and the rest of the increments will be perfectly independent.
This tree construction is an example of branching random walk.
For the model \eqref{eq: X} of zeta, the {\it branching point} $k$ where the paths $X_k(h)$ and $X_k(h')$ roughly decorrelate is
\begin{equation}
\label{eqn: branching}
h\wedge h'=\lfloor \log_2 |h-h'|^{-1}\rfloor\ .
\end{equation}
So $h$ and $h'$ correspond to leaves whose most recent common ancestor is in level $k=h\wedge h'$.
We note that the different nature of the correlations for different ranges of $p$ was already exploited in early work of Hal\'{a}sz \cite{halasz}, although without drawing any connection to branching.

\begin{figure}[h]
\hspace{-2cm}
\begin{subfigure}{0.3\textwidth}
\includegraphics[scale=0.42]{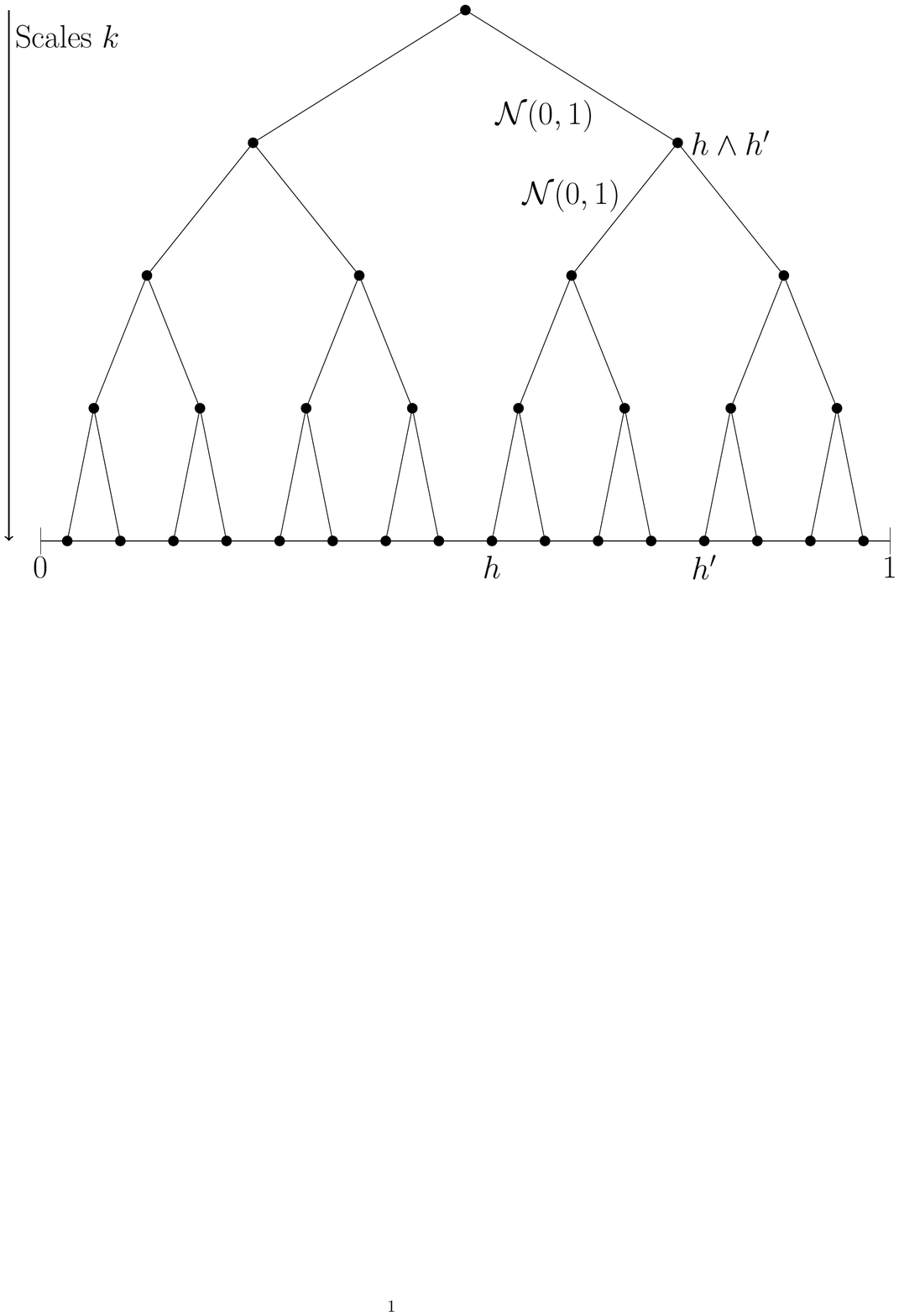}
\vspace{1cm}
\end{subfigure}
\hspace{0.1\textwidth}
\begin{subfigure}{0.5\textwidth}
\includegraphics[scale=0.6]{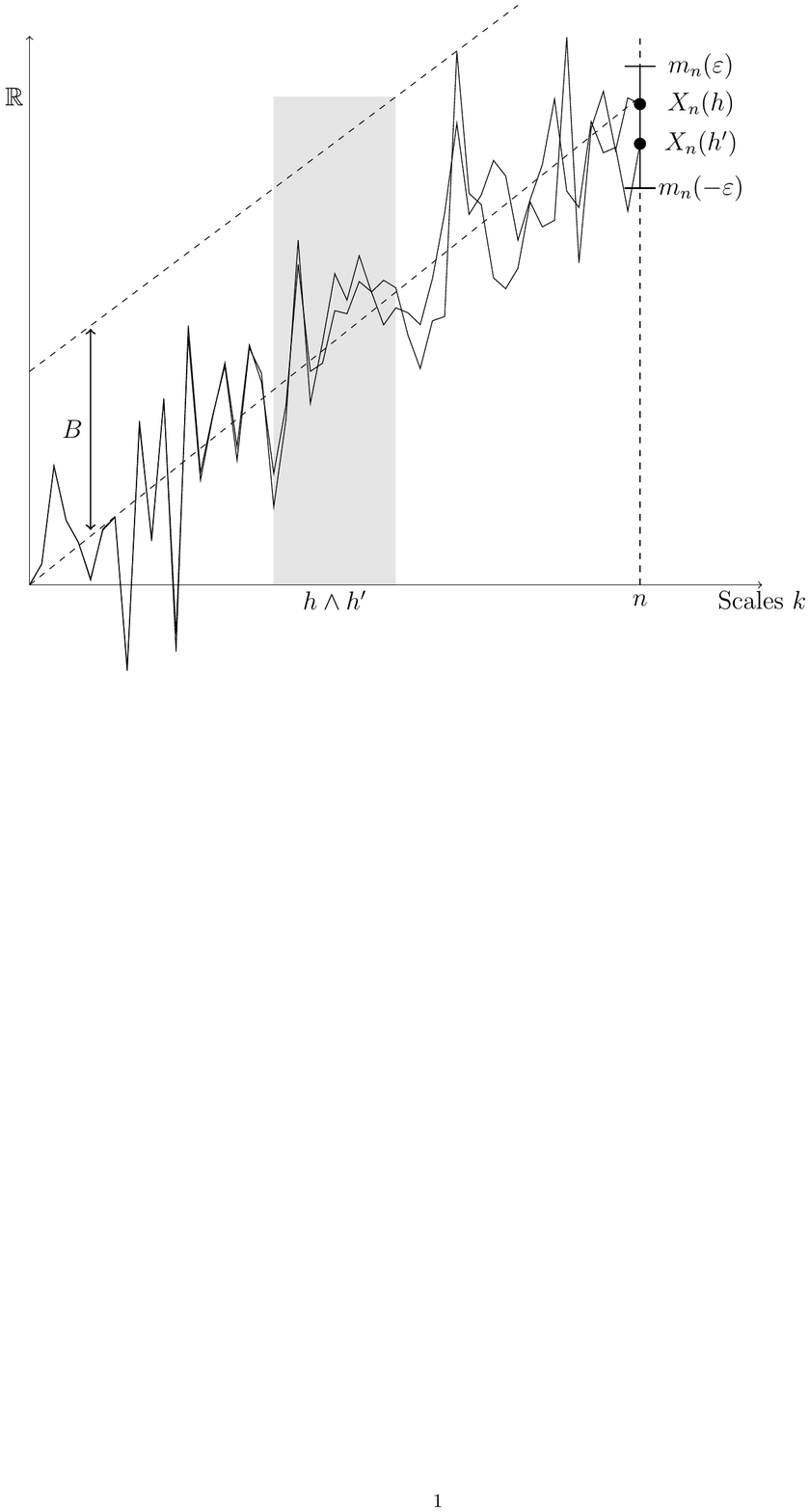}
\end{subfigure}
\vspace{-8.5cm}
\caption{(Left) An illustration of the correlation structure of a branching random walk. 
(Right) A realization of two paths of an approximate branching random walk with increments almost equal before the branching point $h\wedge h'$ and almost decoupled after. 
The barrier below which the paths must stay is also shown.}
\label{fig: tree}
\captionsetup[figure]{position=b}
\end{figure}

A compelling method to analyse the maximum of a branching random walk and of log-correlated processes in general
is a {\it multiscale refinement of the second moment method} as proposed in \cite{kistler}, 
which we implement to the approximate branching setting described above.
Naively, one could first consider the number of variables whose value exceeds a given value $m$, i.e. the {\it number of exceedances},
\begin{equation}
\label{eqn: Z}
Z(m)=\#\{j\leq 2^n: X_n(j/2^n)\geq m\}\ .
\end{equation}
Clearly, $\max_{j\leq 2^n} X_n(j/2^n)\geq m$ if and only if $Z(m)\geq 1$.
Thus an upper bound for the maximum can be obtained by the union bound
\begin{equation}
\label{eqn: markov}
\prob(Z(m)\geq 1)\leq \E[Z(m)]=2^n \prob(X_n(0)\geq m)\ .
\end{equation}
On the other hand, a lower bound can be obtained by the Paley--Zygmund inequality,
\begin{equation}
\label{eqn: PZ}
\prob(Z(m) \geq 1)\geq \frac{\E[Z(m)]^2}{\E[Z(m)^2]}\ .
\end{equation}
More precisely, one would choose $m=m(n)$ large enough in \eqref{eqn: markov} so that $\E[Z(m)]=o(1)$, and $m$ small enough in 
\eqref{eqn: PZ} so that $\E[Z]^{2} =(1+o(1)) \E[Z^2]$, and thus $\prob(Z(m)\geq 1)=1+o(1)$.
For this one needs large deviation estimates: if we think of $X_n(h)$ as Gaussian with variance $n\log 2 /2$, then a standard Gaussian estimate
yields that $P( X_n(h) \ge m)$ is approximately $\frac{\sqrt{n}}{m}e^{-m^2/((\log 2) n)}$.
Thus $2^n P( X_n(0) \ge m)=o(1)$ when
$m=(\log 2)n - \frac{1}{4}\log n + \e \log n$.
This would in fact be the correct answer (the union bound would be sharp) if the random variables $X_n(j/2^n)$ were independent.
However, if $m=(\log 2)n - \frac{3}{4}\log n + \e \log n$, then $2^m\prob( X_n(0) \ge m) \ge cn^{1-\e} \to \infty$, so \eqref{eqn: markov} cannot prove the upper bound we seek in Theorem \ref{thm: main}. Similarly, the right-hand side of \eqref{eqn: PZ} will tend to zero unless $m \le \frac{ \log 2}{2} n$,
since strong correlation between exceedance events for nearby $h, h'$ inflates the second moment. Thus the lower bound
obtained is not close to what we seek even to leading order.

To get tight bounds, one needs to modify the definition of the number of exceedances using an insight
from the underlying approximate tree structure. 
For branching random walk there are exactly $2^k$ distinct partial sums up to the $k$-level, one for each vertex at that level.
By analogy one expects that the ``variation'' in $X_k(h)$ (i.e. in the partial sums up to the $k$-th level) for different $h \in [0,1]$
should be captured by just $2^k$ equally spaced points in $[0,1]$.
Even if they were independent, it would be very unlikely that one of these $2^k$ values exceeded $k\log 2+B$, for $B>0$ growing slowly with $n$,
and it turns out that positive correlations only make it less likely. This can be proved using elementary arguments, cf. Lemma \ref{lem: construction fo barrier}.
In other words,  with high probability, all random walks $X_k(h)$ must lie below the {\it barrier} $k\mapsto k\log 2 +B$.
This suggests to look at the modified number of exceedances
 \begin{equation}
\label{eqn: Z modif}
\widetilde Z(m)=\#\{j\leq 2^n: X_n(j/2^n)\geq m, X_k(j/2^n)<k\log 2 +B\ ,  \forall k\leq n\}\ .
\end{equation}
It turns out that replacing $Z$ by $\widetilde Z$ in the first moment bound \eqref{eqn: markov} and
(with slight modifications) in the second moment bound \eqref{eqn: PZ} will yield the correct answer.
To see this in the former case, we write the first moment by conditioning on the end point:
\begin{equation}
\E[\widetilde Z(m)]=2^n \prob(X_n(0)>m) \times \prob(X_k(0)<k\log 2 +B\ ,  \forall k\leq n| X_n(0)>m)\ .
\end{equation}
By the earlier naive discussion, the first two terms amount to $O(n^{1-\e})$ when we set $m=n\log 2 - \frac{3}{4}\log n + \e \log n$. 
The third term is the probability that a
{\it random walk bridge} starting at $0$ and ending at $m=n\log 2 - \frac{3}{4}\log n + \e \log n$ avoids the barrier $k\log 2 +B$.
This probability turns out to be $n^{-1}$, as shown by the {\it ballot theorem}, cf.~Lemma \ref{lem:ballot theorem}.
Therefore, $\E[\widetilde Z(m)]=O(n^{-\e}) \to 0$, for all $\varepsilon>0$.
A similar analysis can be done for the lower bound, where we have the obvious inequality $\prob(Z(m)\geq 1)\geq \prob(\widetilde Z(m)\geq 1)$.
The extra barrier condition turns out to reduce correlations between exceedance events sufficiently so that the second moment is now essentially
the first moment squared when $m \le n\log 2 - \frac{3}{4}\log n - \e \log n$ (this indicates why the second moment of $Z(m)$ is too large:
in the exponentially unlikely event that a path manages to go far above the barrier, it has exponentially many ``offspring'' that end up far above the
typical level of the maximum).

The form of the subleading correction is thus explained by the extra ``cost'' $n^{-1}$ of satisfying the barrier condition. And the barrier
condition arises because of ``tree-like'' correlations present in the values of (the model of) the zeta function. This suggests the
possibility that the partial sums of the Euler product \eqref{eqn: Euler product taylor} of the actual zeta function behave similarly where the zeta function is large.

To prove Theorem \ref{thm: main}, we must address the imprecisions in the above discussion.
The necessary large deviation estimates are derived in Section \ref{sect: ld}.
The claim that $X_k(h)$ does not vary much below scale $2^{-k}$ is proved in Section \ref{sect: continuity} using a {\it chaining argument}.
Another issue is that our process is not an exact branching random walk because increments are never perfectly independent (for different $h,h'$) nor exactly identical.
To deal with this, we use a Berry--Esseen approximation in Section \ref{sect: gaussian} to show that the random walks are very close to being Gaussian.
This allows for an explicit comparison with Gaussian random walks with i.i.d.~increments and ``perfect'' branching.
Moreover, to get a sharp lower bound with the second moment method, it is necessary to
``cut off the first $r$ scales'' and consider
\begin{equation}
\label{eqn: process cutoff}
X_{r,k}(h)= X_k(h)-X_r(h) \qquad \text{for $h\in \R$, }
\end{equation}
for an appropriately chosen $r$.
Finally, it should be stressed that our approach relies only on controlling first and second moments, which means that the estimates we need
only involve at most two random walks simultaneously.

\section{Preliminaries}
\label{sect: tools}
Throughout the paper, we will write $c$ for absolute constants whose value may change at different occurrences.
A sum over the variable $p$ always denotes a sum over primes.

\subsection{Large Deviation Estimates}
\label{sect: ld}
In this section, we derive the large deviation properties of the increments $(Y_k(h), h\in [0,1])$ and their sum.
We first derive basic facts on their distribution and  in particular on their correlations.

Recall that the random variables  $(U_p, p \text{ primes})$ are i.i.d.~and uniform on the unit circle.
For simplicity, we denote the $p$-th term of the sum over primes in \eqref{eq: main} by,
\begin{equation}\label{eq: W_p def}
  W_p(h)=\frac{\re (U_pp^{-ih})}{p^{1/2}}, h \in \mathbb{R}.
\end{equation}
Note that the law of the process $(W_p(h), h\in \R)$ is translation-invariant on the real line and also invariant under the reflection $h \mapsto -h$. 
A straightforward computation using the law of the $U_p$'s and translation invariance gives
\begin{equation}
\label{eq: variance and covariance of Wps}
 \E\left[W_p(h)W_p(h') \right] = \frac{1}{2p} \cos(|h-h'|\log p) , \text{ for all $h,h'$. }
\end{equation}

In this notation, the increments defined in \eqref{eqn: Y} are
\begin{equation}
\label{eq: Y_k def}
Y_k(h)=\sum_{2^{k-1}< \log p \leq 2^k} W_p(h), h \in \mathbb{R}.
\end{equation}
Using \eqref{eq: variance and covariance of Wps} and the independence of the $U_p$'s, the variance of $Y_k(h)$ becomes
\begin{equation}
\label{eqn: variance}
\sigma_k^2=\text{Var}(Y_k(h))=\sum_{2^{k-1} <\log p \leq 2^k}
\frac{1}{2p}\ ,
\end{equation}
and the covariance of $Y_k(h)$ and $Y_k(h')$ is,
\begin{equation}
 \rho_k(h,h')=\E[Y_{k}(h) Y_{k}(h')] = \sum_{2^{k-1}<\log p\le2^k} \frac{1}{2p}\cos(|h-h'|\log p).
\end{equation}

The next lemma formalizes \eqref{roughcorr}, giving bounds for how close the variance of the increments is to
\begin{equation}\label{eq: changed variance}
  \sigma^2 = (\log 2)/2\ ,
\end{equation}
and for $h\neq h'$, how close the covariance is to the variance before the "branching point" $h\wedge h'$, defined in \eqref{eqn: branching}, 
and how fast it decays after. 
\begin{lem}
\label{lem: primes}
For $h,h'\in \R$ and $k\geq 1$, 
\begin{equation} 
\label{eqn: sigmak estimate}
\sigma_k^2=\E \left[ Y_{k}(h)^{2} \right] = \sigma^2 + O\left(e^{-c\sqrt{2^{k}}}\right),
\end{equation}

\begin{equation}
\label{eqn: correlation estimates}
\rho_k(h,h')=\E \left[ Y_{k}(h) Y_{k}(h') \right]  = 
\begin{cases}
\sigma^2 + O\left( 2^{-2(h\wedge h'-k)} \right) + O\left(e^{-c\sqrt{2^{k}}}\right) &\text{ if $k\leq h\wedge h'$,}\\
O\left( 2^{-(k- h\wedge h')} \right)  &\text{ if $k> h\wedge h'$.}
\end{cases}
\end{equation}
\end{lem}
Note that in both cases the error term decays exponentially in $k$.
\begin{proof}
We use a strong form of the Prime Number Theorem (see Theorem 6.9 of \cite{montgomery-vaughan-multiplicative-nt})
which states that
\begin{equation}\label{eq: PNT}
 \#\{p \le x: p \mbox{ prime}\} = \int_2^{x} \frac{1}{\log u}du + O( x e^{-c \sqrt{ \log x }} ).
\end{equation}
By replacing the sum $\sum_{P<p\le Q}\frac{1}{p}$ with the integral $\int_P^Q \frac{1}{u\log u}du$ 
using \eqref{eq: PNT} and integration by parts, one obtains
$$
\sum_{P<p\le Q}\frac{1}{p} = \log\log Q-\log\log P +O(e^{-c\sqrt{\log P}}), \mbox{ for all }2 \le P \le Q.
$$
This together with \eqref{eqn: variance} yields
\eqref{eqn: sigmak estimate}. Similarly \eqref{eq: PNT} implies that
$$
  \rho_k(h,h') = \frac{1}{2} \int_{e^{2^{k-1}}}^{e^{2^k}} \frac{\cos( |h-h'|\log u )}{u \log u} du + O\left( (1+|h-h'|)e^{-c\sqrt{2^{k-1}}} \right).
$$
When $2^{k}|h-h'|=2^{k-h\wedge h'}\leq 1 $, the claim \eqref{eqn: correlation estimates} follows by using that $\cos(|h-h'|\log u) = 1 + O(|h-h'|^2(\log u)^2)$.
When $2^{-k} |h- h'|^{-1}=2^{-k+h\wedge h'}<1$, we use integration by parts. After the change of variable $v=\log u$, the integral becomes
$$
\frac{\sin(|h-h'| v)}{ |h-h'| v}\Big|_{2^{k-1}}^{2^k}
+
\int_{2^{k-1}}^{2^k} \frac{\sin( |h-h'|v )}{ |h-h'|v^2} du\ .
$$
Both terms are $O(2^{-k} |h-h'|^{-1})$.

\end{proof}

\begin{rem}
 A similar but easier argument using \eqref{eq: PNT} shows that
       \begin{equation}\label{eq: logp over p sum}
         \sum_{P<p\le Q}\frac{\left(\log p\right)^{m}}{p} = O((\log Q)^m), \mbox{ for all } 1 \le P \le Q.
       \end{equation}
\end{rem}

The main results of this section are explicit expressions for the cumulant generating functions of the increments, from which we will deduce large deviation estimates.
For fixed $h, h'\in \mathbb{R}$, we will often drop the dependence on $h$ and $h'$ when it is clear from context and define
$$
\vec{Y}_k= \big( Y_k(h),Y_k(h')\big) \ .
$$
The covariance matrix of $\vec{Y}_k$ is then denoted by
$$
\vec{\Sigma}_{k}=\text{Cov}(\vec{Y}_k)=\left(\begin{array}{cc}\sigma_k^2 & \rho_k \\ \rho_k & \sigma_k^2\end{array}\right)\ .
$$
The eigenvalues of $\Sigma_k$ are $\sigma_k^2\pm \rho_k$. 

The cumulant generating functions are
\begin{equation}
\label{eqn: psi}
\begin{aligned}
\psi^{(1)}_k(\lambda)=\log \E[\exp (\lambda Y_k )]\qquad
\psi^{(2)}_k(\vec{\lambda})= \log \E[\exp ( \vec{\lambda}\cdot \vec{Y}_k ) ]\ ,
\end{aligned}
\end{equation}
where $\lambda\in \R$, $\vec{\lambda}\in\R^2$ and $``\cdot"$ is the inner product in $\R^2$.
The following change of measure will also be needed in the proof of Theorem \ref{thm: main}:
\begin{equation}
\label{eqn: RN1 and RN2}
\frac{d\mathbb{Q}_{\lambda}}{d\mathbb{P}}=
\prod_{k=1}^n \frac{e^{\lambda Y_k}}{e^{\psi^{(1)}_k(\lambda)}}\ \mbox{ for }\lambda \in \mathbb{R},
 \quad \quad \quad
\frac{d\mathbb{Q}_{\vec{\lambda}}}{d\mathbb{P}}=
 \prod_{k=1}^n
\frac{e^{\vec{\lambda}\cdot \vec{Y}_k}}{e^{\psi^{(2)}_k(\vec{\lambda})}}\ \mbox{ for } \vec{\lambda}\in \R^2.
\end{equation}
Recall that in the univariate case,
\begin{equation}
\label{eqn: deriv Q1}
\mathbb Q_\lambda[Y_k]= \frac{d}{d\lambda}\psi_k^{(1)}(\lambda), \qquad \text{Var}_{\mathbb Q_\lambda}(Y_k)=  \frac{d^2}{d\lambda^2}\psi_k^{(1)}(\lambda) \ ,
\end{equation}
and in the multivariate case,
\begin{equation}
\label{eqn: deriv Q2}
\mathbb Q_{\vec{\lambda}}[\vec{Y}_k]= \nabla\psi_k^{(2)}(\vec{\lambda}), \qquad \text{Cov}_{\mathbb Q_{\vec{\lambda}}}(\vec{Y}_k)=  \text{Hess} \ \psi_k^{(2)}(\vec{\lambda})\ .
\end{equation}
The results also provide bounds on these quantities. 
We first state the result for the univariate case. The proof is omitted since it is a special case of the multivariate
bound in Proposition \ref{prop: two point ld}. 
\begin{prop}
\label{prop: one point cgf}
Let $C>0$. For all $0<\lambda<C$ and $k$ large enough (depending on $C$),
the cumulant generating function $\psi_{k}^{\left(1\right)}\left(\lambda\right)$ satisfies
\begin{equation}
\label{eqn: one point cgf}
\psi_{k}^{\left(1\right)}\left(\lambda\right)=\frac{\lambda^2\sigma_k^2}{2}+O\left(e^{-2^{k-1}}\right)\ .
\end{equation}
Moreover, for such $k$, the measure $\mathbb{Q}_{\lambda}$ in \eqref{eqn: RN1 and RN2} satisfies
\begin{equation}
\label{eqn: one point mean and variance}
\mathbb{Q}_{\lambda}\left[Y_{k}\right]=\lambda \sigma_k^2 +O\left(e^{-2^{k-1}}\right)
,\qquad 
Var_{\mathbb{Q}_{\lambda}}\left[Y_{k}\right]=\sigma_{k}^{2}+O\left(e^{-2^{k-1}}\right)\ .
\end{equation}
\end{prop}
One useful consequence of the proposition is a one-point large deviation estimate,
which after being strengthened to a bound
for the maximum over a small interval, will be a crucial
input to the proof of the upper bound of Theorem \ref{thm: main} (see \eqref{eqn: chaining} and \eqref{eq: after tail bound}).
Recall from \eqref{eqn: process cutoff} that $X_{r,k}(h) = X_{k}(h) - X_{r}(h) = \sum_{l=r+1}^{k} Y_{l}(h)$.
\begin{cor}
\label{cor: ld one point}
Let $C>0$. For any $0 \leq r \leq k-1$, $0<x<C(k-r)$ and $h\in \mathbb{R}$,
\begin{equation}
\label{eq: one point ld}
\mathbb{P}\left[X_{r,k}\left(h\right)> x\right]\leq c \exp\left(-\frac{x^{2}}{2(k-r)\sigma^2 } \right),
\end{equation}
where the constant $c$ depends on $C$. 
\end{cor}
\begin{proof}
Using the exponential Chebyshev's inequality, the probability in (\ref{eq: one point ld})
is bounded above by $\exp\left(\sum_{l=r+1}^{k}\psi_{l}^{\left(1\right)}\left(\lambda\right)-\lambda x\right)$, for all $\lambda>0$.
By Proposition \ref{prop: one point cgf} (with, say, $10C$ in place of $C$), we get that if $\lambda \le 10C$, 
$$
\begin{aligned}
\mathbb{P}\left[X_{r,k}\left(h\right)> x\right]
& \le \exp\left( c + \frac{\lambda^{2}}{2}\sum_{l=r+1}^{k}\sigma_{l}^{2}-\lambda x +O(e^{-c2^r})\right)\\
 & \le c\exp\left(\frac{\lambda^{2}}{2}\sum_{l=r+1}^{k}\sigma_{l}^{2}-\lambda x \right)
\overset{\eqref{eqn: sigmak estimate}}{\le} c\exp\left(\frac{\lambda^{2}}{2}\left(k-r\right)\sigma^2-\lambda x \right),
\end{aligned}
$$
where for $l$ too small for \eqref{eqn: one point cgf} to be applied, we simply use that $\psi_{l}(\lambda)$ is bounded.
Setting $\lambda=x\big((k-r)\sigma^2\big)^{-1} \le 10C$ gives the result.
\end{proof}

We now prove the bounds in the multivariate case.
\begin{prop}
\label{prop: two point ld}
Let $C>0$. For all $\vec{\lambda}=\left(\lambda,\lambda^{'}\right)$, where $0<\lambda,\lambda'<C$,
and $k$ large enough (depending on $C$), the cumulant generating function $\psi_{k}^{\left(2\right)}\left(\vec{\lambda}\right)$ satisfies
\begin{equation}
\label{eq: two point cgf}
\psi_{k}^{(2)}\big(\vec{\lambda}\big)=\frac{1}{2}\vec{\lambda}\cdot\vec{\Sigma}_{k}\vec{\lambda}+O\big(e^{-2^{k-1}}\big).
\end{equation}
Moreover, for such $k$, the measure $\mathbb{Q}_{\vec{\lambda}}$ in \eqref{eqn: RN1 and RN2} satisfies
\begin{equation}
\label{eq: two point mean and variance}
\mathbb{Q}_{\vec{\lambda}}\left[\vec{Y}_{k}\right]=\vec{\Sigma}_{k}\vec{\lambda}+O\big(e^{-2^{k-1}}\big)\mbox{ and }
\text{\normalfont Cov}_{\mathbb{Q}_{\lambda}}\left[\vec{Y}{}_{k}\right]=\vec{\Sigma}_{k}+O\big(e^{-2^{k-1}}\big).
\end{equation}
\end{prop}

\begin{proof}
We first compute
\begin{equation}
\begin{array}{rcl}\label{eq: psiW}
  \psi_p^{W}( \vec{\lambda} ) & =& \log \mathbb{E}[ \exp( \lambda W_p(0) + \lambda' W_p( |h - h'| ) ]\\
                              & = & \log \frac{1}{2\pi}\int_0^{2\pi} \exp\left( \frac{\lambda}{p^{1/2}} \cos( \theta ) + \frac{\lambda'}{p^{1/2}} \cos( \theta + |h-h'|\log p)\right) d\theta.
\end{array}
\end{equation}
Recall that for any $a,b \in \mathbb{R}$,
\begin{equation}\label{eq: Bessel Identity}
 \frac{1}{2\pi}\int_0^{2\pi} \exp( a\cos(\theta) + b\sin(\theta) ) d\theta = I_0( \sqrt{a^2 + b^2} ),
\end{equation}
where $I_n$ denotes the $n$-th modified Bessel function of the first kind. The identity $\cos( \theta + \eta) = \cos(\theta)\cos(\eta) - \sin(\theta) \sin( \eta )$ can be used
together with \eqref{eq: Bessel Identity} to write the integral in the bottom line of \eqref{eq: psiW} as
\begin{equation}
\begin{array}{l}\label{eq: Using Bessel Identity}
I_0\left( \sqrt{ \frac{1}{p} \left(\lambda+\cos(|h-h'|\log p)\lambda'\right)^2 + \frac{1}{p} \left(\sin(|h-h'|\log p) \lambda'\right)^2 } \right)\\
= I_0\left( \sqrt{ \frac{1}{p} \left( \lambda^2 + 2\lambda\lambda'\cos(|h-h'|\log p) +\lambda'^2 \right) }\right) = I_0\left( \sqrt{ 2\vec{\lambda}\cdot \vec{M}_p \vec{\lambda} } \right),
\end{array}
\end{equation}
where
$$
\vec{M}_p=
\frac{1}{2p}\left(\begin{array}{cc} 1& \cos(|h-h'|\log p)\\ \cos(|h-h'|\log p) & 1 \end{array}\right)\ ,
$$
is the covariance matrix of $( W_p(h),W_p(h'))$, see \eqref{eq: variance and covariance of Wps}. Thus writing
\begin{equation}\label{eq: f def}
  f(x) = \log I_0(\sqrt{2x}),
\end{equation}
we have $\psi_p^{W}( \vec{\lambda} ) = f\left( \vec{\lambda}\cdot \vec{M}_p \vec{\lambda} \right)$.
 Recall
that $I_0(x)$ has Taylor expansion $I_0(x) = 1 + \frac{x^2}{4} + \frac{x^4}{64} + O(x^6)$
(which can be verified by expanding in \eqref{eq: Bessel Identity}), so that $f$ has Taylor expansion
\begin{equation}\label{eq: f taylor exp}
f(x) = \frac{x}{2} - \frac{x^2}{16} + O(x^3).
\end{equation}

Now since the random variables $U_p$ are independent,
$$
\psi^{(2)}_k\big(\vec{\lambda}\big)
= \sum_{2^{k-1}<\log p \leq 2^k} \psi_p^{W}( \vec{\lambda} )
= \sum_{2^{k-1}<\log p \leq 2^k} f\left( \vec{\lambda}\cdot \vec{M}_p \vec{\lambda} \right).
$$
The bound \eqref{eq: f taylor exp} implies that for $k$ large enough (depending on $C$),
$$
\psi^{(2)}_k\big(\vec{\lambda}\big)
= \sum_{2^{k-1}<\log p \leq 2^k} \left( \frac{1}{2}\vec{\lambda}\cdot \vec{M}_p \vec{\lambda} + O\left(p^{-2}\right) \right)
= \frac{1}{2}\vec{\lambda}\cdot \vec{\Sigma}_k \vec{\lambda} + O\left(e^{-2^{k-1}}\right).
$$
This proves \eqref{eq: two point cgf}.

The first claim of \eqref{eq: two point mean and variance} follows similarly after noting that the gradient
of the map $\vec{\lambda} \to f( \vec{\lambda}\cdot \vec{M}_p \vec{\lambda} )$ is 
$\vec{M}_p \vec{\lambda} f'\left( \vec{\lambda}\cdot \vec{M}_p \vec{\lambda} \right)$, and using the bound
$f'(x) = \frac{1}{2} + O(x)$, valid for $x\in[0,1]$. Finally the second claim of \eqref{eq: two point mean and variance}
follows by noting that the Hessian of the aforementioned map is 
$$
  \vec{M}_p f'(\vec{\lambda}\cdot \vec{M}_p)
  + (\vec{M}_p \vec{\lambda})(\vec{M}_p \vec{\lambda})^T f''\left( \vec{\lambda}\cdot \vec{M}_p \vec{\lambda} \right),
$$
and using the previous bound for $f'(x)$, and that $f''(x)$ is bounded in $[0,1]$.
\end{proof}

\subsection{Continuity estimates}
\label{sect: continuity}
The main result of this section is a maximal inequality which shows that the maximum
over an interval of length $2^{-k}$ of the field $X_{r,k}(h)$ is close to the value of
the field at the mid-point of the interval, where $X_{r,k}(h)$ is defined in \eqref{eqn: process cutoff}.
One of the upshots is to reduce the proof of
the upper bound of the maximum of the process on $[0,1]$ to an upper bound on the maximum over a discrete set of points in Section \ref{sect: upper}.

\begin{prop}
\label{prop: max inequality}
Let $C>0$. For any $0 \le r \leq k-1$, $0\le x\le C \left(k-r\right)$, $2 \le a \le 2^{2k}-x$ and $h \in \mathbb{R}$, 
\begin{equation}
\label{eqn: chaining}
\mathbb{P}
\left[\max_{h':\left|h'-h\right|\le2^{-k-1}}X_{r,k}(h')> x+a,X_{r,k}\left(h\right)\le x\right]\le c\exp\left(-\frac{x^{2}}{2\left(k-r\right)\sigma^2}-ca^{3/2}\right),
\end{equation}
where the constants $c$ depend on $C$.
\end{prop}
The proof of the proposition is postponed until the end of the section.
It is based on a {\em chaining argument} and an estimate on joint large deviations of $X_{r,k}(h)$ and of the difference $X_{r,k}(h')-X_{r,k}(h)$ for $|h'-h|\leq 2^{-k-1}$, see Lemma \ref{lem: chaining LD} below. 
The exponent of the $a$ term is probably not optimal.
A direct consequence of the proposition is the following large deviation bound of the maximum of $X_{k}\left(h\right)$ over an interval of length $2^{-k}$.
\begin{cor}
\label{cor: tail of sup over interval}
Let $C>0$. For any $0 \le r \leq k-1$, $h\in \mathbb{R}$ and $0\le x\le C\left(k-r\right)$,
\begin{equation}
\mathbb{P}\left[\max_{h^{'}:\left|h^{'}-h\right|\le2^{-k-1}}X_{r,k}(h')> x \right]\le c\exp\left(-\frac{x^{2}}{2\left(k-r\right)\sigma^2}\right),\label{eq: tail of sup over interval}
\end{equation}
where the constant $c$ depends on $C$.
\end{cor}

\begin{proof}
The left-hand side of \eqref{eq: tail of sup over interval} is at most
\[
\mathbb{P}\left[\max_{h':\left|h^{'}-h\right|\le2^{-k-1}}X_{r,k}\left(h'\right)>(x-2)+2,X_{r,k}\left(h\right)\le x-2\right]+\mathbb{P}\left[X_{r,k}\left(h\right)> x-2\right]\ .
\]
The bound follows by \eqref{eqn: chaining} with $x-2$ in place of $x$ and $a=2$, and the bound \eqref{eq: one point ld}.
\end{proof}
\begin{rem}
\label{rem: leading order upper bound}
A union bound over $2^{n}$ intervals of length $2^{-n}$ yields
\begin{equation}
\label{eqn: max int}
\mathbb{P}\left[\max_{h\in[0,1]}X_{n}\left(h\right)\ge (1+\delta)n\log 2\right]\le c2^{-n\delta}\ \mbox{ for all }\delta>0,
\end{equation}
where \eqref{eq: tail of sup over interval} is used with $r=0$ and $k=n$ (note that $X_n(h) = Y_0(h) + X_{0,n}(h)$ and $Y_0$ is bounded).
This proves that $\max_{h\in[0,1]}X_{n}\left(h\right)$ is at most $(1+o(1))n\log 2$, which is tight to leading order,
but does not include the subleading correction present in \eqref{eq: main in terms of X} and \eqref{eqn: m_n definition}.
\end{rem}

To prove Proposition \ref{prop: max inequality} we will use the following 
large deviation estimate for $X_{r,k}\left(0\right)$ and the difference
$X_{r,k}\left(h_{2}\right)-X_{r,k}\left(h_{1}\right)$ (jointly), where $|h_2-h_1| \le 2^{-k}$.
It shows that on a large deviation scale the two quantities are
essentially independent, and that the difference decays rapidly with $|h_2-h_1|$. The latter is
a consequence of the covariance of the field $X_{r,k}\left(h\right)$
losing its log-correlation structure below scale $2^{-k}$, and instead decaying linearly with distance. 
\begin{lem}
\label{lem: chaining LD}
Let $C>0$. For any $0 \leq r \leq k-1$, $0\leq x\leq C\left(k-r\right)$, $ 0 \le y \le2^{2k}$
and any distinct $-2^{-k-1} \leq h_1, h_2 \leq 2^{-k-1}$,
\begin{equation}
\label{eq: three point ld}
\mathbb{P}\left[X_{r,k}\left(0\right)\ge x,X_{r,k}\left(h_{2}\right)-X_{r,k}\left(h_{1}\right)\ge y\right]\le c\exp\left(-\frac{x^{2}}{2\left(k-r\right)\sigma^2}-\frac{ c y^{3/2}}{2^{k}|h_{2}-h_{1}|}\right),
\end{equation}
where the constants $c$ depend on $C$. 
\end{lem}

\begin{proof}
Observe first that we may assume $y$ is bigger than a large constant depending on $C$ times $2^{k}|h_2 - h_1|$, (and therefore also bigger than a large constant times $2^{2k}|h_2-h_1|^{2}$), because otherwise the required bound follows from \eqref{eq: one point ld}.

For any $\lambda_{1},\lambda_{2} > 0$, the left-hand side of \eqref{eq: three point ld}
is bounded above by
\begin{equation}
\label{eqn: 3-point markov}
\mathbb{E}\left[\exp\left(\lambda_{1}X_{r,k}\left(0\right)+\lambda_{2}\left(X_{r,k}\left(h_{2}\right)-X_{r,k}\left(h_{1}\right)\right)\right)\right]\exp\left(-\lambda_{1}x-\lambda_{2}y\right).
\end{equation}
We will show that if $\lambda_{1} \le 10C$ and $1 \le \lambda_{2} \le |h_2-h_1|^{-1}$,
\begin{equation}
\label{eqn: 3-point exp}
\begin{aligned}
&\mathbb{E}\left[\exp\left(\lambda_{1}X_{r,k}\left(0\right)+\lambda_{2}\left(X_{r,k}\left(h_{2}\right)-X_{r,k}\left(h_{1}\right)\right)\right)\right]\\
&\hspace{2cm}
\le c \exp\left(
\frac{\lambda_1^2\sigma^2}{2}(k-r) + c\lambda_2 2^{k}|h_2-h_1| + c( \lambda_2 2^{k}|h_2-h_1|)^2 
\right)\ .
\end{aligned}
\end{equation}
The result then follows by choosing $\lambda_1=x\ ((k-r)\sigma^2)^{-1}$ and $\lambda_2= c y^{1/2} \ 2^{-k}|h_2-h_1|^{-1}$
in \eqref{eqn: 3-point markov} and \eqref{eqn: 3-point exp}, for a suitable small $c$, and using our assumption that $y$ is bigger than a large constant times $2^{k}|h_2 - h_1|$. Note that the assumptions on $x,y, h_1$ and $h_2$ ensure that
$\lambda_1\le10C$ and $1 \le \lambda_2 \le |h_2-h_1|^{-1}$.

We now prove \eqref{eqn: 3-point exp}. First we note that similarly to the argument from \eqref{eq: psiW} to \eqref{eq: Using Bessel Identity}, 
\begin{equation}\label{eq: three point cgf}
\mathbb{E}\left[\exp\left(\lambda_{1}W_{p}\left(0\right)+\lambda_{2}\left(W_{p}\left(h_{2}\right)-W_{p}\left(h_{1}\right)\right)\right)\right],
\end{equation}
can be written explicitly as
\begin{equation}
I_0\left(  \sqrt{ \frac{1}{p} \Big(\lambda_1+(\cos(h_2\log p)-\cos(h_1\log p))\lambda_2\Big)^2 + \frac{1}{p} \Big((\sin(h_2\log p)-\sin(h_1\log p)) \lambda_2\Big)^2 } \right).
\end{equation}
 Recall from \eqref{eq: f taylor exp} that $\log I_0 (\sqrt{x}) = \frac{1}{4}x + O(x^2)$,
and that $\cos(h_2\log p)-\cos(h_1\log p)=O(|h_2-h_1|\log p)$ and $\sin(h_2\log p)-\sin(h_1\log p)=O(|h_2-h_1|\log p)$.
Thus provided $\lambda_1 \le 10C$, $1 \le \lambda_2 \le |h_2-h_1|^{-1}$ and $p$ is large enough, the logarithm of the quantity
in \eqref{eq: three point cgf} is at most
\begin{equation}\label{eq: bound on I0} 
\begin{aligned}
&\frac{1}{4p} ( \lambda_1 + c \lambda_2 |h_2-h_1| \log p)^2 + \frac{c}{p}(\lambda_2 |h_2-h_1| \log p)^2 + cp^{-2} \\
&\hspace{2cm}\le
\frac{\lambda_1^2}{4p} + \frac{c}{p} \lambda_2 |h_2-h_1|\log p
+ \frac{c}{p}\left( \lambda_2 |h_2-h_1|\log p\right)^2  + cp^{-2}.
\end{aligned}
\end{equation}
Here we used the fact that $\lambda_1 \leq 10C$. After summing over $2^{r} < \log p \le 2^k$ we get that
$$
\begin{aligned}
&\log \mathbb{E}\left[\exp\left(\lambda_{1}X_{r,k}\left(0\right)+\lambda_{2}\left(X_{r,k}\left(h_{2}\right)-X_{r,k}\left(h_{1}\right)\right)\right)\right]\\
&\hspace{2cm} \le c + \displaystyle{\sum_{2^{r} < \log p \le 2^k}} \frac{\lambda_1^2}{4p} + c\displaystyle{\sum_{2^{r} < \log p \le 2^k}}\frac{\log p}{p} \lambda_2|h_2-h_1| + c\displaystyle{\sum_{2^{r} < \log p \le 2^k}}\frac{(\log p)^2}{p}\left( \lambda_2|h_2-h_1|\right)^2 .
\end{aligned}
$$
In the above, if $p$ is too small for \eqref{eq: bound on I0} to be an upper bound, we simply use that \eqref{eq: three point cgf} is bounded. 
The claim \eqref{eqn: 3-point exp} now follows from the bounds \eqref{eqn: sigmak estimate} and \eqref{eq: logp over p sum}.
\end{proof}

We are now ready to prove Proposition \ref{prop: max inequality}.
We will use the following notation: for $k\in \N$, let
\begin{equation}
\label{eqn: H}
\begin{array}{c}
\mathcal{H}_k \mbox{ be the set } \frac{1}{2^k}\mathbb{Z} \mbox{ of dyadic rationals, so that }
\mathcal{H}_{0} \subset \mathcal{H}_{1}\subset\ldots\subset\mathcal{H}_{k}\subset\ldots\subset \mathbb{R} \\
\mbox{ is a nested sequence of sets of equally spaced points and } |\mathcal{H}_k \cap [0,1)| = 2^k.
\end{array}
\end{equation}

\begin{proof}[Proof of Proposition \ref{prop: max inequality}]
Without loss of generality, we may assume $h=0$. We can also round $x$ up and decrease $a$ so that
we may assume that $x$ is an integer and $a\ge1$. Define the events
\[
B_{q}=\left\{ X_{r,k}\left(0\right)\in\left[x-q-1,x-q\right]\right\} ,q=0,1,\ldots,x-1,\mbox{ and }B_{x}=\left\{ X_{r,k}\left(0\right)\le0\right\} .
\]
Note that the left-hand side of \eqref{eqn: chaining} is at most 
\begin{equation}\label{eq:sum over p}
\sum_{q=0}^{x}\mathbb{P}\left[B_{q}\cap\left\{ \max_{h^{'}\in A}\left\{X_{r,k}\left(h'\right)-X_{r,k}\left(0\right)\right\}\ge a+q\right\} \right],
\end{equation}
where $A=[-2^{-k-1},2^{-k-1}]$.
Let $(h_i, i\geq 0)$ be a dyadic sequence such that $h_0=0$, $h_i \in \mathcal H_{k+i} \cap A$ and $\lim_{i \to \infty} h_{i} = h'$,
 so that $|h_{i+1} - h_i| \in \{0, 2^{-k-i-1}\}$ for all $i$.
Because the map $h\mapsto X_{r,k}(h)$ is almost surely continuous,
$$
X_{r,k}(h')-X_{r,k}(0)=\sum_{i=0}^{\infty} \big(X_{r,k}(h_{i+1})- X_{r,k}(h_{i})\big)\ .
$$
The right-hand side converges almost surely, since
$\sum_{i=0}^{l} \big(X_{r,k}(h_{i+1})- X_{r,k}(h_{i})\big)\ = X_{r,k}(h_{l+1}) - X_{r,k}(0) \to X_{r,k}(h') - X_{r,k}(0)$,
because  $X_{r,k}(h)$ is continuous almost surely.
Since $\sum_{i=0}^\infty \frac{1}{2(i+1)^2}\leq 1$, we have the inclusion of events,
\[
\left\{ X_{r,k}\left(h'\right)-X_{r,k}\left(0\right)\ge a+q\right\} \subset\bigcup_{i=0}^{\infty}\left\{ X_{r,k}\left(h_{i+1}\right)-X_{r,k}\left(h_{i}\right)\ge\frac{a+q}{2\left(i+1\right)^{2}}\right\} .
\]
This implies that $\left\{ \max_{h^{'}\in A}\left(X_{r,k}\left(h^{'}\right)-X_{r,k}\left(0\right)\right)\ge a+q\right\} $ is included in 
\[
\bigcup_{i=0}^{\infty}\bigcup_{\begin{array}{c}
h_{1}\in\mathcal{H}_{k+i}\cap A ,\\
h_{2}=h_{1} \pm 2^{-k-i-1}
\end{array}}\left\{ X_{r,k}\left(h_{2}\right)-X_{r,k}\left(h_{1}\right)\ge\frac{a+q}{2\left(i+1\right)^{2}}\right\} ,
\]
where we have ignored the case $h_1=h_2$ since then event 
$\left\{ X_{r,k}\left(h_{2}\right)-X_{r,k}\left(h_{1}\right)\ge\frac{a+q}{2\left(i+1\right)^{2}}\right\}$
is the empty set.
Because $\left|\mathcal{H}_{k+i}\cap A\right|\le c2^{i}$, 
the $q$-th summand in \eqref{eq:sum over p} is at most,
$$
\sum_{i=0}^{\infty} c2^{i}\sup_{\begin{array}{c}
h_{1}\in\mathcal{H}_{k+i}\cap A ,\\
h_{2}=h_{1} \pm 2^{-k-i-1}
\end{array}}\mathbb{P}\left[B_{q}\cap\left\{ X_{r,k}\left(h_{2}\right)-X_{r,k}\left(h_{1}\right)\ge\frac{a+q}{2\left(i+1\right)^{2}}\right\} \right].\label{eq: done markov over dyadic level}
$$
Note that $a+q\leq a+x\leq 2^{2k}$ by assumption. The inequality \eqref{eq: three point ld} can thus be applied to get that \eqref{eq: done markov over dyadic level} is at most
$$
c\sum_{i=0}^{\infty}2^{i}\exp\left(-\frac{\left(x-q-1\right)^{2}}{2(k-r)\sigma^2}-c2^{i}\frac{\left(a+q\right)^{3/2}}{(i+1)^3}\right)
\le ce^{-\frac{\left(x-q-1\right)^{2}}{2(k-r)\sigma^2}-c(a+q)^{3/2}}.
$$
Since $e^{-c\left(a+q\right)^{3/2}}\le e^{-ca^{3/2}-cq^{3/2}}$, \eqref{eq:sum over p} is thus at most
$$
\begin{aligned}
ce^{-ca^{3/2}}\sum_{q=0}^{x}e^{-\left(x-q-1\right)^{2}/\left(2(k-r)\sigma^2\right)-cq^{3/2}}
&\le ce^{-x^{2}/\left(2(k-r)\sigma^2\right)-ca^{3/2}}\sum_{q=0}^{x}e^{c\left(q+1\right)-cq^{3/2}}\\
&\le ce^{-x^{2}/\left(2(k-r)\sigma^2\right)-ca^{3/2}},
\end{aligned}
$$
where we used the assumption $x\leq C(k-r)$. This proves \eqref{eqn: chaining}. 
\end{proof}

\subsection{Gaussian approximation}
\label{sect: gaussian}
The purpose of this section is to compare the increments $Y_{k}(h)$ to Gaussian random variables with
mean and variance independent of $k$, both for a single $h\in\mathbb{R}$ and for
vectors $(Y_{k}(h_1),Y_{k}(h_2))$ for $h_1 \ne h_2 \in \mathbb{R}$.
This will be used in the subsequent sections to apply the ballot theorem and derive bounds on the probability that $X_{r,k}(h_1)$ and $X_{r,k}(h_2)$ satisfy
a barrier condition.
One reason to pass to Gaussian random variables is that the standard ballot theorem provides such bounds for random walks with i.i.d.~increments.
It does not immediately apply to the process $k \mapsto X_{r,k}(h)$, whose increments $Y_k(h)$ have slightly different distributions for different $k$.
Moreover, we need to show that the increments $Y_{k}(h_1)$ and $Y_{k}(h_2)$ for two points $h_1\ne h_2$ become roughly independent
when $k$ is beyond the branching point $h_1\wedge h_2$, cf.~\eqref{eqn: branching}.
To quantify this, we introduce a parameter $\Delta$ and refer to the scale $h_1\wedge h_2+\Delta$ as the {\it decoupling point}.
Passing to Gaussian variables facilitates the proof of the decoupling, since in the Gaussian case we can investigate independence
solely by controlling the covariance and the mean.

Our main tool is the following multivariate Berry--Esseen approximation for independent random vectors.
For the remainder of the paper, $\eta_{\vec \mu, \Sigma}$ will denote the Gaussian measure with mean vector $\vec \mu$ and covariance matrix $\Sigma$.
\begin{lem}[Corollary 17.2 in \cite{bhattacharya-rao}, see also Theorem 1.3 in \cite{goetze}]
\label{lem: normal approx}
Let $(\vec W_j, j\geq1 )$ be a sequence of independent random vectors on $(\R^d,\mathcal B(\R^d),P)$
with mean $E\left[\vec W_j\right]$ and covariance matrix $\text{Cov}(\vec W_j)$. 
Define 
$$
\vec{\mu}_m = \sum_{j=1}^m E\left[\vec W_j\right]  \mbox{ and } \Sigma_m=\sum_{j=1}^m \text{Cov}(\vec W_j)\ .
$$
Let $\lambda_m$ be the smallest eigenvalue of $\Sigma_m$ and $Q_m$ be the law of $\vec W_1+\dots + \vec W_m$. 

There exists an absolute constant $c$ depending only on the dimension $d$ such that
$$
\sup_{A\in\mathcal A} \Big| Q_m(A ) - \eta_{\vec \mu_m,\Sigma_m}(A)\Big| \leq c \lambda_m^{-3/2} \sum_{j=1}^m E[\|\vec W_j - E[\vec W_j ]\|^3]\ .
$$
where $\mathcal A$ is the collection of Borel measurable convex subsets of $\R^d$. 
\end{lem}
Before stating the results, we recall the notation from Section \ref{sect: ld}: \
 $\mathbb{Q}_{\vec{\lambda}}$ is the product measure from \eqref{eqn: RN1 and RN2} and
for fixed $h_1, h_2\in \mathbb{R}$, we write $\vec{Y}_{k}=\left(Y_{k}\left(h_{1}\right),Y_{k}\left(h_{2}\right)\right)$,
$\vec{X}_{r,n}=\left(X_{r,n}\left(h_{1}\right),X_{r,n}\left(h_{2}\right)\right)$.
We show that beyond the decoupling point $h_1\wedge h_2+\Delta$, the increments under $\mathbb{Q}_{\vec{\lambda}}$ are close (in terms of $\Delta$) to being independent Gaussians
with mean $\lambda \sigma^2$ and variance $\sigma^2 = (\log 2)/2$.
\begin{prop}
\label{prop: two point gaussian comparison}
Let $\lambda \in \R$ and $\Delta > 0$. Let $h_{1},h_{2}\in \mathbb{R}$, $m\geq h_1\wedge h_2\ +\Delta$ and $\mu=\lambda \sigma^2$.
For any convex subsets $A_{k}\subseteq\mathbb{R}^{2},k=m+1,\ldots,n$, we have
\begin{equation}
\begin{aligned}
&\mathbb{Q}_{\vec{\lambda}}\left[\vec{X}_{m,k}\in A_{k}\ \forall m<k\leq n\right]\\
&=\big(1+O(e^{-c\Delta})\big)\  \eta_{\mu,\sigma^{2}}^{\times2\left(n-m\right)}\left\{ \vec{y}\in \mathbb{R}^{2\times(n-m)}:\ \sum_{j=1}^{k}\vec{y}_{j}\in A_{k+m}\ \forall k=1,\dots,n-m\right\}+O(e^{-e^{c\Delta}}) \ ,
\end{aligned}
\end{equation}
where $\eta_{\mu,\sigma^{2}}^{\times2\left(n-m\right)}$ denotes the product measure (on $2(n-m)$ independent Gaussians each with mean $\mu$ and variance $\sigma^2$).
\end{prop}

\begin{proof}
Recall that $\vec{Y}_{k}=\sum_{2^{k-1}<\log p\leq2^{k}}\vec{W}_{p}$
where $\vec{W}_{p}=\left(W_{p}\left(h_{1}\right),W_{p}\left(h_{2}\right)\right)$.
The proof has two steps. First, Lemma \ref{lem: normal approx} is applied successively for each $k$ from $k=n$ down to $k=m+1$ to pass to a Gaussian measure.
Then we explicitly compare the resulting Gaussian measure $\otimes_{k=m+1}^n\eta_{\tilde{\vec{\mu}}_k,\tilde \Sigma_k}$
(the product of $(n-m)$ bivariate Gaussians with means $\tilde{\vec\mu}_{k}=\mu_k\ (1,1)=\big(\mathbb{Q}_{\vec{\lambda}}\left[Y_{k}(h_1)\right],\mathbb{Q}_{\vec{\lambda}}\left[Y_{k}(h_2)\right]\big)$ and covariance matrices
$\tilde \Sigma_k=\mbox{Cov}_{\mathbb{Q}_{\vec{\lambda}}}\left[\vec{Y}_{k}\right]$), to the decoupled measure $\eta_{\mu,\sigma^{2}}^{\times2\left(n-m\right)}$. 

Conditioning on the values of $\vec{Y}_{j}$ for all $m+1 \leq j \leq n-1$, then
applying Lemma \ref{lem: normal approx} to the $\vec{W}_{p}$ with $2^{n-1}<\log p\leq2^{n}$,
and finally integrating over $\vec{Y}_{j}$  we obtain 
\begin{equation}
\begin{aligned}
\label{eq:berry-eseen 2D}
&\Big|\mathbb{Q}_{\vec{\lambda}}\left[\vec{X}_{m,k}\in A_{k}\ \forall m<k\leq n\right]\\
&-\mathbb{Q}_{\vec{\lambda}}\times\eta_{\tilde{\vec\mu}_{n},\tilde{\Sigma}_{n}}\left[\sum_{j=m+1}^{k}\vec{Y}_{j}\in A_{k}\ \forall m<k\leq n-2,
\sum_{j=m+1}^{n-1}\vec{Y}_{j}\in A_{n-1}\cap(A_{n}-\vec{y}_{n})\right] \Big|\\
&\le c\lambda_{n}^{-3/2}{\displaystyle \sum_{2^{n-1}<\log p\leq2^{n}}}\mathbb{Q}_{\vec{\lambda}}\left[\|\vec{W}_{p}-\mathbb{Q}_{\vec{\lambda}}\left[\vec{W}_{p}\right]\|^{3}\right],
\end{aligned}
\end{equation}
where $\vec{y}_{n}$ is sampled from $\eta_{\tilde{\vec\mu}_{n},\tilde{\Sigma}_{n}}$,  $\lambda_{n}$ is the smallest eigenvalue of $\tilde{\Sigma}_{n}$, and $A_{n}-\vec{y}_{n}$ is the set $A_{n}$ translated by $\vec{y}_{n}$.
Since an intersection of convex sets is convex, the lemma can be applied in the same way to the $\vec{W}_{p}$'s contributing to
$\vec{Y}_{n-1}$, $\vec{Y}_{n-2}$, and so on. The resulting estimate is then
\begin{equation}
\label{eq:berry-esseen 2d error term}
\begin{aligned}
&\Big |\mathbb{Q}_{\vec{\lambda}}\left[\vec{X}_{m,k}\in A_{k}\ \forall m<k\leq n\right]
-\otimes_{k=m+1}^{n}\eta_{\tilde{\vec \mu}_{k},\tilde{\Sigma}_{k}}\left\{ \vec{y}\in \mathbb{R}^{2\times(n-m)}:\ \sum_{j=m+1}^{k}\vec{y}_j \in A_{k}\ \forall k=m+1,\dots,n \right\} \Big|\\
& \hspace{2cm} \le c \sum_{k=m+1}^{n}\sum_{2^{k-1}<\log p\leq2^{k}}\lambda_{k}^{-3/2}\mathbb{Q}_{\vec{\lambda}}\left[\|\vec{W}_{p}-\mathbb{Q}_{\vec{\lambda}}\left[\vec{W}_{p}\right]\|^{3}\right].
\end{aligned}
\end{equation}
For $k>h_1\wedge h_2 + \Delta$, the eigenvalues $\lambda_k$ are uniformly bounded away from $0$.
Indeed, observe that by \eqref{eq: two point mean and variance}, and the discussion preceding \eqref{eqn: psi}, and Lemma \ref{lem: primes},
$$
\lambda_{k} = \sigma_{k}^{2} - \rho_{k} + O\left(e^{-2^{k-1}}\right) = \sigma^2+O(e^{-c\sqrt{2^k}} + e^{-c\Delta})\ge c> 0,
$$
for $\Delta$ large enough but fixed. Also by construction, the norm of the vector $\vec{W}_{p}$ is bounded by $cp^{-1/2}$. 
Hence the error term in \eqref{eq:berry-esseen 2d error term} is bounded by 
\begin{equation}
\label{eqn: gaussian error 1}
c\sum_{2^{m}<\log p\leq2^{n}}p^{-3/2}\le c e^{-2^{m-1}} \le e^{-e^{c\Delta}}.
\end{equation}

It remains to compare the measure $\otimes_{k=m+1}^{n} \eta_{\tilde{\vec\mu}_{k}, \tilde \Sigma_k}$ with the measure $\eta_{\mu,\sigma^2}^{\times 2(n-m)}$.
The specifics of the considered event play no role at this point, so we write $B$ for a generic measurable subset of $\R^2$. 
We show
\begin{equation}
\label{eqn: gaussian error 2}
\eta_{\tilde{\vec\mu}_{k}, \tilde \Sigma_k}[B]
=\big(1+O(e^{-c(k-h_1\wedge h_2)})\big) \ \eta^{}_{\mu,\sigma^2}[B] + O(e^{-e^{c(k-h_1\wedge h_2)}})\ , \text{ $\forall k>m$.}
\end{equation}
Together with \eqref{eqn: gaussian error 1} and \eqref{eq:berry-esseen 2d error term}, this implies the proposition since the estimate \eqref{eqn: gaussian error 2} can be applied successively integrating in each coordinate to get for any $A\subseteq \R^{2(n-m)}$
$$
\begin{aligned}
\otimes_{k=m+1}^n\eta_{\tilde{\vec\mu}_{k}, \tilde \Sigma_k}[A]
&=\prod_{k=m+1}^n \big(1+O(e^{-c(k-h_1\wedge h_2)})\big) \eta^{\times 2(n-m)}_{\mu,\sigma^2}[A] + \sum_{k=m+1}^nO(e^{-e^{c(k-h_1\wedge h_2)}})\\
&=\big(1+O(e^{-c\Delta})\big)  \eta^{\times 2(n-m)}_{\mu,\sigma^2}[A] + O(e^{-e^{c\Delta}})\ .
\end{aligned}
$$

To prove \eqref{eqn: gaussian error 2}, we compare densities.
Proposition \ref{prop: two point ld} and Lemma \ref{lem: primes} give
\begin{equation}
\label{eqn: approx mu sigma}
\begin{aligned}
\mu_{k} &=\mu +O(2^{-(k-h_1\wedge h_2)}), \qquad
\tilde \Sigma_k&= \sigma^2 \mathbbm{1} +O(2^{-(k-h_1\wedge h_2)})\ ,
\end{aligned}
\end{equation}
where $\mathbbm{1}$ is the $2\times 2$ identity matrix, using that $k>m>h_1\wedge h_2+\Delta$.
Consider the set,
$$
E_k=\{\vec{y}\in \R^2: \| \vec{y}-\tilde{\vec\mu}_{k}\| \leq  2^{(k-h_1\wedge h_2)/4}\}\ .
$$
A straightforward Gaussian estimate yields
$$
\eta_{\tilde{\vec\mu}_{k}, \tilde \Sigma_k}[E_k^c]\leq  \exp\left(-c\frac{2^{(k-h_1\wedge h_2)/2}}{\sigma^2} \right)\leq e^{-e^{c(k-h_1\wedge h_2)}} \ ,
$$
and similarly for $\eta_{\mu, \sigma^2}^{\times 2}[E_k^c]$. Therefore, it suffices to prove \eqref{eqn: gaussian error 2} for $B \subset E_k$.
The density of $\eta_{\tilde{\vec{\mu}}_k, \tilde{\Sigma}_k}$ with respect to Lebesgue measure is,
\begin{equation}\label{eq: bivariate normal density}
 \frac{1}{2\pi (\mbox{det} \tilde{\Sigma}_k)^{1/2} }
e^{ - (\vec{y}-\tilde{\vec\mu}_{k})\cdot \tilde \Sigma_k^{-1}(\vec{y}-\tilde{\vec\mu}_{k})/2 }.
\end{equation}
By \eqref{eqn: approx mu sigma},
$$
(\det   \tilde \Sigma_k)^{-1/2}=\sigma^{-2} \big(1+O(2^{-(k-h_1\wedge h_2)})\big)\ .
$$
Furthermore for all $\vec{y}\in\R^2$,
$$
(\vec{y}-\tilde{\vec\mu}_{k})\cdot \tilde \Sigma_k^{-1}(\vec{y}-\tilde{\vec\mu}_{k})
= \sigma^{-2}\|\vec{y}-\tilde{\vec\mu}_{k}\|^2+(\vec{y}-\tilde{\vec\mu}_{k})\cdot (\tilde \Sigma_k^{-1}-\sigma^{-2}\mathbbm{1}) (\vec{y}-\tilde{\vec\mu}_{k})\ .
$$
By \eqref{eqn: approx mu sigma} and the definition of $E_k$, the error term is
$$
(\vec{y}-\tilde{\vec\mu}_{k})\cdot (\tilde \Sigma_k^{-1}-\sigma^{-2}\mathbbm{1}) (\vec{y}-\tilde{\vec\mu}_{k})=O(2^{-(k-h_1\wedge h_2)/4})\ .
$$
Thus, on $E_k$, the density \eqref{eq: bivariate normal density} equals $\big(1+O(e^{-c(k-h_1\wedge h_2)})\big)\frac{1}{2\pi \sigma^2}e^{-||\vec{y}-\tilde{\vec\mu}_{k}||^{2}/2}$.
In particular,
$$\eta_{\tilde{\vec\mu}_{k}, \tilde \Sigma_k}[B]
=\big(1+O(e^{-c(k-h_1\wedge h_2)})\big) \ \eta^{\times 2}_{\tilde{\mu}_k,\sigma^2}[B]\ \mbox{ for any } B \subset E_k.
$$
It remains to compare the densities of $\eta_{\tilde{\mu}_k,\sigma^2}$ and $\eta_{\mu,\sigma^2}$.
We have that
$$
(y-\tilde\mu_k)^2=(y-\mu)^2+(\tilde\mu_k-\mu)^2 -2(y-\mu)(\tilde\mu_k-\mu)\ .
$$
The second term is $O(2^{-(k-h_1\wedge h_2)})$ by \eqref{eqn: approx mu sigma}.
The third term can be estimated using the fact that $|y-\mu_k|=O(2^{(k-h_1\wedge h_2)/4})$:
$$
|(y-\mu)(\tilde \mu_k-\mu)|
\leq (|y-\tilde \mu_k|+|\mu_k-\mu|)|\mu_k-\mu|
=O(2^{-3(k-h_1\wedge h_2)/4})\ .
$$
This implies that on $B \subset E_k$
$$
\eta^{\times 2}_{\tilde{\mu}_k,\sigma^2}[B]
=\big(1+O(e^{-c(k-h_1\wedge h_2)})\big) \ \eta^{\times 2}_{\mu,\sigma^2}[B]\ .
$$
This concludes the proof of the claim  \eqref{eqn: gaussian error 2}.
\end{proof}
The next proposition provides a Gaussian comparison before the branching point. The proof is omitted, as it follows the previous
one closely, with $\mu$ replaced by $2 \lambda \sigma^2$ in \eqref{eqn: approx mu sigma}.
\begin{prop}
\label{prop: two point gaussian comparison before branch}
Let $\lambda \in \R$ and $\Delta > 0$. Let $h_{1},h_{2}\in \mathbb{R}$, $m\leq h_1\wedge h_2\ - \Delta$ and $\mu=2\lambda \sigma^2$.
For any convex subsets $A_{k}\subseteq\mathbb{R}^2,k=m+1,\ldots,n$, we have
\begin{equation}
\begin{aligned}
&\mathbb{Q}_{\vec{\lambda}}\left[\vec{X}_{m,k}(h_1) \in A_{k}\ \forall m<k\leq n\right]\\
&=\big(1+O(e^{-c\Delta})\big)\  \eta_{\mu,\sigma^{2}}^{\times 2\left(n-m\right)}\Big\{ \vec{y}\in \mathbb{R}^{\times2(n-m)}:\ \sum_{j=1}^{k}y_{j}\in A_{k+m}\ \forall k=1,\dots,n-m\Big\}+O(e^{-e^{c\Delta}}) \ .
\end{aligned}
\end{equation}
\end{prop}

A one-point Gaussian approximation for the
measure $\mathbb{Q}_{\lambda}$ from  \eqref{eqn: RN1 and RN2}
 will also be needed. The proof is again similar to the proof of Proposition \ref{prop: two point gaussian comparison} and is omitted.
One noticeable difference is in \eqref{eqn: approx mu sigma} where the covariance estimate is
replaced by $\sigma_k^2=\sigma^2+O(e^{-e^{ck}})$ because of \eqref{eqn: sigmak estimate}.
The additive error $e^{-e^{c\Delta}}$ is then replaced by $e^{-e^{cm}}$. The multiplicative
error $1+O(e^{-c\Delta})$ becomes $1+O(e^{-e^{cm}})$, and can thus be ``absorbed'' in the additive error.

\begin{prop}
\label{prop: one point gaussian comparison}
Let $\lambda \in \R$, $h\in \mathbb{R}$, $0 \le m < n$
and  $\mu=\lambda \sigma^2$.
For any convex subsets $A_{k}\subseteq\mathbb{R},k=m+1,\ldots,n$, we have
\begin{equation}
\begin{aligned}
&\mathbb{Q}_{\lambda}\left[X_{m,k}(h)\in A_{k}\ \forall m<k\leq n\right]\\
&\hspace{1cm} = \eta^{\times (n-m)}_{\mu,\sigma^{2}}\left\{ y\in \mathbb{R}^{\times(n-m)}:\ \sum_{j=1}^{k}y_{j}\in A_{k+m}\ \forall k=1,\dots,n-m\right\}+O(e^{-e^{cm}}) \ .
\end{aligned}
\end{equation}
\end{prop}

\subsection{Ballot theorem}
The ballot theorem provides an estimate
for the probability that a random walk stays below a certain value
and ends up in an interval. We state the case we need, which is that of Gaussian random walk with increments
of mean $0$ and variance $\sigma^2$.
\begin{lem}
\label{lem:ballot theorem}
Let $\left(X_{n}\right)_{n\ge0}$ be a
Gaussian random walk with increments of mean $0$ and variance $\sigma^2 > 0$, with $X_0=0$. Let $\delta>0$.
 There is a constant $c=c(\sigma,\delta)$ such that for all $a>0$, $b \le a - \delta$ and $n\ge1$
\begin{equation}\label{eq:ballot theorem ub}
  P\left[X_{n}\in\left(b,b+\delta\right)\mbox{ and }X_{k}\le a \mbox{ for }0<k<n\right] \le c\frac{(1+a)(1+a-b)}{n^{3/2}}.
\end{equation}
Also provided $\delta<1$,
\begin{equation}\label{eq:ballot theorem lb}
  \frac{1}{cn^{3/2}} \le P\left[X_{n}\in\left(0,\delta\right)\mbox{ and }X_{k}\le 1 \mbox{ for }0<k<n\right].
\end{equation}
\end{lem}
\begin{proof}
Note that $(X_{k})_{0\le k\le n}$ has the law of $(\sigma B_{ k })_{0\le k\le n}$, where $(B_t)_{t\ge0}$ is standard Brownian motion.
Thus we see that the probability in \eqref{eq:ballot theorem ub} conditioned on $X_{n}=y$ can be written as the probability
that a Brownian bridge avoids a barrier at integer times. The bound (6.4) of \cite{WebbbExactAsymptoticsofFreezingTransitionOfLogCorrelatedREM} shows, after shifting by $a/\sigma$ and reflecting,
that this condidional probability is at most $c(1 + a/\sigma)(1 + (a-b-\delta)/\sigma)/n$. Noting that $P[ X_n \in (b,b+\delta) ] \le cn^{-1/2}$
then yields \eqref{eq:ballot theorem ub}. In a similar fashion the display below (6.4) in \cite{WebbbExactAsymptoticsofFreezingTransitionOfLogCorrelatedREM} gives \eqref{eq:ballot theorem lb}.
\end{proof}

\section{Proof of Theorem \ref{thm: main}}\label{sect: proof}
In this section, we prove \eqref{eq: main in terms of X}, that is,
\begin{equation}\label{eq: main in terms of X 2}
  \lim_{n \to \infty} \prob\left[ m_n(-\varepsilon) \le \max_{h\in [0,1]} X_n(h) \le m_n(\varepsilon) \right] = 1 \mbox{, for all }\varepsilon>0.
\end{equation}
This proves Theorem \ref{thm: main} for the subsequence $T=e^{2^n}$, $n\in\N$.
The extension of the argument to general sequences $T$ follows by trivial adjustments.
We will need to consider the process $X_{r,n}(h)$ with the first $r$ scales cutoff, see \eqref{eqn: process cutoff}.
Throughout this section we use
\begin{equation}\label{eq: def of r}
r=\lfloor\left(\log\log n\right)^{2}\rfloor.
\end{equation}
First we show that the difference between $\max_{h\in[0,1]}X_{r,n}\left(h\right)$ and $\max_{h\in[0,1]}X_{n}\left(h\right)$
is negligible compared to the subleading correction term.
\begin{lem}
\label{lem: taking care of up to r}
For all $\varepsilon>0$,
\begin{eqnarray}
 \lim_{n\to\infty}\mathbb{P}\left[\max_{h\in [0,1]}X_{n}\left(h\right)\ge m_{n}\left(2\varepsilon\right),\max_{h\in[0,1]}X_{r,n}\left(h\right)\le m_{n-r}(\varepsilon)\right]=0,\label{eq: taking care of up to r UB}\\
\lim_{n\to\infty}\mathbb{P}\left[\max_{h\in [0,1]}X_{n}\left(h\right)\le m_{n}\left(-2\varepsilon\right),\max_{h\in[0,1]}X_{r,n}\left(h\right)\ge m_{n-r}\left(-\varepsilon\right)\right]=0.\label{eq: taking care of up to r LB}
\end{eqnarray}
\end{lem}
\begin{proof}
The event in the probability in \eqref{eq: taking care of up to r UB}
implies $\max_{h\in[0,1]}X_{r}\left(h\right)\ge\left(\log2\right)r+\varepsilon\log\left(n-r\right)\ge100\left(\log2\right)r$,
where the last inequality holds for $n$ large enough. 
But \eqref{eqn: max int}, with $n=r$, gives
\[
\mathbb{P}\left[\max_{h\in[0,1]}X_{r}\left(h\right)\ge100(\log 2) r\right]\leq 2^{-99r}\to0,\mbox{ as }r\to\infty.
\]
Since the laws of $\max_{h\in[0,1]}X_{r}\left(h\right)$
and $-\min_{h\in[0,1]}X_{r}\left(h\right)$, coincide we
also have that the probability
$
\mathbb{P}\left[\min_{h\in[0,1]}X_{r}\left(h\right)\le-100(\log 2)r\right]
$
tends to $0$ as $r\to\infty$, which similarly implies (\ref{eq: taking care of up to r LB}).
\end{proof}

In the proof of \eqref{eq: main in terms of X 2} we will use a change of measure under which the process $X_{r,n}$
has an upward drift of
\begin{equation}
  \mu(\varepsilon) = \frac{m_{n-r}(\varepsilon)}{n-r} = \frac{(n-r)\log 2 - \frac{3}{4} \log(n-r) + \epsilon \log(n-r)}{n-r} .
\end{equation}
We use the following consequence of \eqref{eqn: m_n definition} and \eqref{eq: changed variance} several times,
\begin{equation}
\label{eq: mu squared}
 \frac{\mu(\varepsilon)^2}{2\sigma^2} = \log2 - \left( \frac{3}{2} - 2\varepsilon \right) \frac{\log(n-r)}{n-r} + o(n^{-1}).
\end{equation}

\subsection{Proof of the upper bound}
\label{sect: upper}

In this section we prove the upper bound part of \eqref{eq: main in terms of X 2}.
By Lemma \ref{lem: taking care of up to r}, it suffices to prove the following upper bound for $\max_{h\in[0,1]}X_{r,n}\left(h\right)$.

\begin{prop}
\label{prop: Upper bound}
For all $\varepsilon>0$,
\begin{equation}
  \lim_{n\to\infty}\mathbb{P}\left[\max_{h\in[0,1]}X_{r,n}\left(h\right)\ge m_{n-r}(\varepsilon)\right]=0.
\end{equation}
\end{prop}

The first step is to reduce the proof to a bound on the maximum over the discrete set $\mathcal H_n \cap [0,1]$
(as defined in \eqref{eqn: H})
using the continuity estimates from Section \ref{sect: continuity}.
\begin{lem}
\label{lem: continuity at level n}
For all $\varepsilon>0$,
\begin{equation}
\lim_{n\to\infty}\mathbb{P}\left[\max_{h\in [0,1]}X_{r,n}\left(h\right)\ge m_{n-r}\left(2\varepsilon\right),\max_{h\in\mathcal{H}_{n} \cap [0,1] }X_{r,n}\left(h\right)\le m_{n-r}(\varepsilon)\right]=0.\label{eq: continuity at level n}
\end{equation}
\end{lem}
\begin{proof}
Using translation invariance and a union bound on $2^n$ intervals, the probability in \eqref{eq: continuity at level n} is at most
$$
  2^{n}\mathbb{P}\left[\max_{h: |h|\leq 2^{-n-1}}X_{r,n}\left(h\right)\ge m_{n-r}\left(2\varepsilon\right),X_{r,n}\left(0\right)\le m_{n-r}(\varepsilon)\right].\label{eq:cont union bound}
$$
Proposition \ref{prop: max inequality} can be applied with $k=n$, $x=m_{n-r}(\varepsilon)=(n-r)\mu(\varepsilon)$ and $a=m_{n-r}(2\varepsilon)-m_{n-r}(\varepsilon)=\varepsilon \log{(n-r)} <2^{2n}-x$.
This gives the upper bound
\begin{equation}\label{eq: after tail bound}
  c2^{n}\exp\left( -\left(n-r\right) \frac{\mu(\varepsilon)^2}{2\sigma^2}-c\varepsilon^{3/2}\left(\log\left(n-r\right)\right)^{3/2}\right).
\end{equation}
Using \eqref{eq: mu squared} and \eqref{eq: def of r}, we get that \eqref{eq: after tail bound} is at most
$$
  c2^{n} \left( 2^{r-n} \left(n-r\right)^{\frac{3}{2}-2\varepsilon}e^{-c\varepsilon^{3/2}\left(\log\left(n-r\right)\right)^{3/2}} \right)
  = o\left(1\right).
$$
\end{proof}

The second step is to show that for each $h\in \mathcal [0,1]$ the process $k\to X_{r,k}\left(h\right)$
satisfies a barrier condition with very high probability. This simply requires a union bound together with continuity estimates.
\begin{lem}
\label{lem: construction fo barrier}
For all $\varepsilon>0$,
\begin{equation}
\lim_{n\to\infty}\mathbb{P}\left[\exists h\in[0,1],k\in\left\{ \lfloor \log n\rfloor^2,\ldots,n\right\} \mbox{ s.t. }X_{r,k}\left(h\right)> (k-r)\mu(\varepsilon) +\left(\log n\right)^{2}\right]=0.\label{eq:barrier}
\end{equation}
\end{lem}
\begin{proof}
By two successive union bounds, first over the scales $k=\lfloor \log n \rfloor^2,\ldots,n$, and then, for each of those scales, over $2^{k}$
intervals (together with translation invariance), the probability in \eqref{eq:barrier} is at most
$$
  \sum_{k= \lfloor \log n\rfloor^2}^{n}2^{k}
  \mathbb{P}\left[\max_{h:\left|h\right|\le2^{-k-1}}X_{r,k}\left(h\right)\ge (k-r)\mu(\varepsilon) +\left(\log n\right)^{2}\right].\label{eq:markovoverlevels}
$$
The maximal inequality \eqref{eq: tail of sup over interval} can be applied since the right-hand side of the inequality in
the probability is less than a constant times $(k-r)$.
Thus the sum is bounded above by,
\[
c\sum_{k= \lfloor \log n\rfloor^2}^{n}2^{k}\exp\left(-\frac{\left((k-r)\mu(\varepsilon)+\left(\log n\right)^{2}\right)^{2}}{2\left(k-r\right)\sigma^2}\right).
\]
Using \eqref{eq: mu squared} the argument in the exponential is at least
$$
 \left(k-r\right)\log 2 -\frac{3}{2}\log\left(n-r\right)+c\left(\log n\right)^{2} \ .
$$
We conclude that the probability in \eqref{eq:barrier} is at most
\[
c\sum_{k= \lfloor \log n\rfloor^2}^{n}2^{k} \left( 2^{r-k} n^{3/2}e^{-c\left(\log n\right)^{2}} \right)=c2^{r}n^{5/2}e^{-c\left(\log n\right)^{2}}=o\left(1\right).
\]
\end{proof}
Lemma \ref{lem: continuity at level n} and Lemma \ref{lem: construction fo barrier} show that $\max_{h\in [0,1]}X_{r,n}\left(h\right)$
exceeds $m_{n-r}\left(2\varepsilon\right)$ only if, for some $h\in\mathcal{H}_{n} \cap [0,1]$, $X_{r,n}\left(h\right)$
exceeds $m_{n-r}(\varepsilon)$ and the process $(X_{r,k}\left(h\right), \lfloor \log n \rfloor^2 \leq k\leq n)$
stays below a linear barrier. The number of $h\in\mathcal{H}_{n}$ that manage this feat is
\begin{equation}
\begin{array}{c}
Z^{+}=\sum_{h\in\mathcal{H}_{n}\cap[0,1]}\textbf{1}_{J^{+}\left(h\right)}, \mbox{ where }\\
J^{+}\left(h\right)=\left\{ X_{r,n}\left(h\right)\ge m_{n-r}(\varepsilon)\mbox{, }X_{r,k}\left(h\right)\le (k-r)\mu(\varepsilon) 
+\left(\log n\right)^{2}\forall k\ge\lfloor\log n\rfloor^2 \right\} .
\end{array}
\end{equation}
We show $\mathbb{P}\left[Z^{+}>0\right] \le c2^{r}\left(\log n\right)^{6}\left(n-r\right)^{-2\varepsilon}$, thereby proving Proposition \ref{prop: Upper bound} since the right-hand side is $o(1)$ by the definition \eqref{eq: def of r} of $r$.
Here we shall use the previous Gaussian approximation results and the ballot theorem.
\begin{prop}
\label{prop: Upper bound one point expectation}
For all $\varepsilon>0$,
\begin{equation}
\label{eq: upper bound one point expectation}
\mathbb{P}\left[Z^{+}>0\right]\le\mathbb{E}\left[Z^{+}\right] \le c2^{r}\left(\log n\right)^{6}\left(n-r\right)^{-2\varepsilon} \ .
\end{equation}
\end{prop}
\begin{proof}
By translation invariance and linearity of expectation, we have $\mathbb{E}\left[Z^{+}\right]=2^{n}\mathbb{P}\left[J^{+}\left(0\right)\right]$.
We show that
\begin{equation}
\label{eq: upper bound one point to show}
\mathbb{P}\left[J^{+}(0)\right]\le c2^{r-n}\left(\log n\right)^{6}\left(n-r\right)^{-2\varepsilon},
\end{equation}
thus yielding (\ref{eq: upper bound one point expectation}). 
To prove \eqref{eq: upper bound one point to show}, let
$\lambda=\mu(\varepsilon)/\sigma^2$, and recall the definition of $\mathbb{Q}_\lambda$ from \eqref{eqn: RN1 and RN2}.
We have that
\begin{equation}
\label{eq: change of measure bound}
  \mathbb{P}\left[J^{+}(0)\right] \le \mathbb{Q}_{\lambda}\left[J^{+}(0)\right]e^{\sum_{k=r+1}^{n}\psi_{k}^{\left(1\right)}\left(\lambda\right)-\lambda (n-r)\mu(\varepsilon)},
\end{equation}
because $X_{r,n}\left(0\right)\ge (n-r)\mu\left(\varepsilon\right)$ on the event $J^{+}\left(0\right)$. 
Using the estimates \eqref{eqn: one point cgf} and \eqref{eqn: sigmak estimate} we get that
\begin{equation}
\label{eqn: 1 point rate}
\begin{aligned}
\sum_{k=r+1}^{n}\psi_{k}^{\left(1\right)}\left(\lambda\right)-\lambda (n-r)\mu(\varepsilon)
& = -(n-r)\frac{\mu(\varepsilon)^2}{2\sigma^2} + O\left(e^{-c\sqrt{2^r}}\right).\\
\end{aligned}
\end{equation}
By \eqref{eq: mu squared}, the exponential in \eqref{eq: change of measure bound} is thus at most $c2^{r-n} (n-r)^{\frac{3}{2} - 2\varepsilon}$.
It remains to show
\begin{equation}
\mathbb{Q}_{\lambda}\left[J^{+}(0)\right]\le c\left(\log n\right)^{6}\left(n-r\right)^{-3/2}.\label{eq: suffices to show}
\end{equation}
The event $J^{+}(0)$ takes the form in Proposition \ref{prop: one point gaussian comparison} with $m=r$.
Thus $\mathbb{Q}_{\lambda}\left[J^{+}(0)\right]$ is at most
$\eta^{\times(n-r)}_{\mu(\varepsilon),\sigma^2}(E_1)+O(e^{-e^{cr}})$, where
$$
 E_1 = \Big\{ \vec{y}\in\mathbb{R}^{n-r}:\sum_{l=1}^{k}(y_{l}-\mu(\varepsilon))\le  \left(\log n\right)^{2}\ \forall k\geq \lfloor\log n\rfloor^2-r,\sum_{l=1}^{n-r}(y_{l}-\mu(\varepsilon)) \ge 0 \Big\}.
$$
After recentering, the probability of $E_1$ is simply
\begin{equation}\label{eq: after rece}
  \eta^{\times(n-r)}_{0,\sigma^2}\Big\{ \vec{y}\in\mathbb{R}^{n-r}:\sum_{l=1}^{k}y_{l}\le\left(\log n\right)^{2} \forall k\geq  \lfloor\log n\rfloor^2-r,\sum_{l=1}^{n-r}y_{l}\ge 0 \Big\} \ .
\end{equation}
By conditioning on $\sum_{l=1}^{  \lfloor\log n\rfloor^2 -r }y_{l} =q$, we may bound the above by the supremum over $q\in [-(\log n)^2,(\log n)^2]$ of 
$\eta^{\times(n-(\log n)^2)}_{0,\sigma^2}(E_2) + O( ce^{-(\log n)^2} )$, where
\begin{equation}\label{eq: after sup}
 E_2 = \big\{ \vec{y}\in\mathbb{R}^{n-(\log n)^2}:\sum_{l=1}^{k}y_{l}\le\left(\log n\right)^{2}-q \ \forall k\ge0,\sum_{l=1}^{n- \lfloor\log n\rfloor^2}y_{l}\ge -q \big\}\ .
\end{equation}
This is because of the standard Gaussian bound
$$
  \eta^{\times((\log n)^2 -r )}_{0,\sigma^2}\Big\{ \vec{y}\in\mathbb{R}^{(\log n)^2 -r } : \sum_{l=1}^{ (\log n)^2 -r }y_{l} \le -(\log n)^2 \Big\} \le c\exp\left(-c\frac{(\log n)^4}{(\log n)^2-r} \right).
$$
For a given $q$, the probability of the event in \eqref{eq: after sup} may be bounded above by a union bound over a partition of $[-q, (\log n)^2 -q]$ into intervals of length $1$, and
the ballot theorem (Lemma \ref{lem:ballot theorem}). This gives an upper bound for \eqref{eq: after rece} of
\[
\sup_{ -(\log n)^2 \le q \le (\log n)^2 } (\log n)^2 \times c\frac{\left( 1 + \left(\log n\right)^{2} - q \right) \left(2\left(\log n\right)^{2}\right)}{\left(n-r\right)^{3/2}}
\le
c(\log n)^6\left(n-r\right)^{-3/2}.
\]
This proves \eqref{eq: suffices to show}, and thus also \eqref{eq: upper bound one point to show}
and \eqref{eq: upper bound one point expectation}.
\end{proof}

\subsection{Proof of the lower bound}

In this section, we prove the lower bound part of \eqref{eq: main in terms of X 2}.
The proof is reduced to a lower bound on $\max_{h\in[0,1]}X_{r,n}\left(h\right)$ by Lemma \ref{lem: taking care of up to r}.
We show:
\begin{prop}
\label{prop: lower bound}
For all $\varepsilon>0$,
\begin{equation}\label{eq: lower bound}
  \lim_{n\to\infty}\mathbb{P}\left[\max_{h\in[0,1]}X_{r,n}\left(h\right)\geq m_{n-r}\left(-\varepsilon\right)\right]=1\ .
\end{equation}
\end{prop}

As for the upper bound, we consider a modified number of exceedances with a barrier.
For $\delta>0$, let
$$
J^{-}(h)=\left\{ X_{r,n}(h) \in [m_{n-r}(-\varepsilon), m_{n-r}(-\varepsilon)+\delta], \ X_{r,k}(h)\leq (k-r)\mu(-\varepsilon) +1 \ \forall k=r+1,\dots, n\right\}\ .
$$
We omit the dependence on the parameter $\delta$ in the notation for simplicity. Consider the random variable,
$$
Z^{-}=\sum_{h\in\mathcal H_n \cap [0,1) } \textbf{1}_{J^{-}(h)}\ .
$$
Clearly, $\max_{h\in[0,1]}X_{r,n}\left(h\right)\geq m_{n-r}\left(-\varepsilon\right)$ if and only if $Z^{-}\geq 1$.
The Paley--Zygmund inequality implies that
$$
\prob(Z^{-}\geq 1)\geq \frac{\E[Z^{-}]^2}{\E[(Z^{-})^2]}.
$$
We will prove the following estimates for the first and second moments of $Z^{-}$.
Let
\begin{equation}
\label{eqn: A}
A=\Big\{\vec{y}\in \R^{n-r}: \sum_{k=1}^{n-r} y_k  \in [0,\delta], \ \sum_{k=1}^{l-r} y_k \leq 1\ \forall l=r+1,\dots, n\Big\}.
\end{equation}
\begin{lem}
For $\delta>0$,
\label{lem: first moment}
\begin{equation}
\E[Z^{-}]\geq \left(1+o\left(1\right)\right) e^{-c\delta} \ 2^r (n-r)^{\frac{3}{2}+2\varepsilon}\  \eta_{0,\sigma^2}^{\times(n-r)}[A]\ .
\end{equation} 
\end{lem}

\begin{lem}
\label{lem: second moment}
For $\delta>0$,
\begin{equation}
\label{eq:second moment bound}
\mathbb{E}\left[(Z^{-})^{2}\right]\le (1+o(1))\left(2^{r}\left(n-r\right)^{\frac{3}{2}+2\varepsilon} \eta_{0,\sigma^2}^{\times\left(n-r\right)}\left[A\right]\right)^{2}\ .
\end{equation}
\end{lem}
The lower bound \eqref{eq: lower bound} follows directly from these two lemmas.
\begin{proof}[Proof of Proposition \ref{prop: lower bound}]
By the Paley--Zygmund inequality, Lemma \ref{lem: first moment} and Lemma \ref{lem: second moment}, we have
$$
\mathbb{P}\left[\max_{h\in[0,1]}X_{r,n}\left(h\right)\geq m_{n-r}\left(-\varepsilon\right)\right]\geq \prob(Z^{-}\geq 1)\geq
 \frac{\E[Z^{-}]^2}{\E[\left(Z^{-}\right)^2]}
 \geq 
(1+o(1)) e^{-2c\delta}\ .
$$
The result follows by taking the limits $n\to\infty$, then $\delta\to 0$.
\end{proof}
We now prove the bound on $\mathbb{E}[Z^{-}]$.
\begin{proof}[Proof of Lemma \ref{lem: first moment}]
Translation invariance implies $\E[Z^{-}]=2^n \ \prob[J^{-}(0)]$.
Consider the probability $\mathbb{Q}_{\lambda}$ from \eqref{eqn: RN1 and RN2}, where $\lambda = \mu(-\varepsilon)/\sigma^2$.
(By \eqref{eqn: one point mean and variance} and \eqref{eqn: sigmak estimate}, this choice of $\lambda$ implies that
$\mathbb{Q}_{\lambda}\left[ Y_k(0) \right]$ is approximately $\mu(-\varepsilon)$.)
Since on the event $J^{-}\left(0\right)$ we have that $X_{r,n}\le (n-r)\mu(-\varepsilon)+\delta$,
the definition of $\mathbb{Q}_{\lambda}$ implies that
\begin{equation}
\label{eqn: one point 1}
\mathbb{P}\left[J^{-}(0)\right]
\ge
\mathbb{Q}_{\lambda}\left[J^{-}\left(0\right)\right]e^{\sum_{k=r+1}^{n}\psi_{k}^{(1)}(\lambda)-\lambda (n-r) \mu\left(-\varepsilon\right) - c\delta}\ .
\end{equation}
Proceeding as in \eqref{eqn: 1 point rate} to estimate $\sum_{k=r+1}^{n}\psi_{k}^{(1)}(\lambda)-\lambda (n-r) \mu\left(-\varepsilon\right)$, and using \eqref{eq: mu squared}, we get
$$
\mathbb{P}\left[J^{-}(0)\right]\geq  \left(1+ o(1)\right) e^{-c\delta} 2^{-(n-r)} (n-r)^{3/2+2\varepsilon}\mathbb{Q}_{\lambda}\left[J^{-}(0)\right]\ .
$$
The event  $J^{-}(0)$ is of the form appearing in the Berry--Esseen approximation of Proposition \ref{prop: one point gaussian comparison}.
The result can be applied with $m=r$, and after recentering the increments by their mean $\mu = \lambda \sigma^2 = \mu(-\varepsilon)$ we get,
$$
\mathbb{Q}_{\lambda}\left[J^{-}(0)\right]=\eta^{\times(n-r)}_{0, \sigma^2}\left[A\right]+O\left(e^{-e^{cr}}\right)\ .
$$
Note that \eqref{eq:ballot theorem lb} of the ballot theorem (Lemma \ref{lem:ballot theorem}) ensures that
\begin{equation}\label{eq: A lb from ballot}
  \eta^{\times(n-r)}_{0, \sigma^2}\left[A\right] \ge c(n-r)^{-3/2}.
\end{equation}
Thus $\eta^{\times(n-r)}_{0, \sigma^2}\left[A\right]$ dominates $e^{-e^{cr}}$, since $r=\lfloor\left(\log\log n\right)^{2}\rfloor$.
This proves the lemma.
\end{proof}
\begin{rem}
We note for future reference that the same reasoning (using that $X_{r,n}\geq (n-r)\mu(-\varepsilon)$ on $J^{-}(0)$, cf. \eqref{eqn: one point 1})
gives the upper bound,
\begin{equation}
\label{eqn: 1 point rate upper}
\mathbb{P}\left[J^{-}(0)\right]\leq  \left(1+ o(1)\right)  2^{-(n-r)} (n-r)^{3/2+2\varepsilon}\eta^{\times(n-r)}_{0, \sigma^2}\left[A\right]\ .
\end{equation}
\end{rem}
To prove the second moment bound in Lemma \ref{lem: second moment} we use the identity
\begin{equation}
\label{eqn: second moment two-point}
\E[(Z^{-})^2]=\sum_{h_{1},h_{2}\in\mathcal{H}_{n} \cap [0,1) }\mathbb{P}\left[J^{-}\left(h_{1}\right)\cap J^{-}\left(h_{2}\right)\right]\ .
\end{equation}
We thus seek bounds on $\mathbb{P}\left[J^{-}\left(h_{1}\right)\cap J^{-}\left(h_{2}\right)\right]$ for $h_1 \ne h_2$.
This is the key additional difficulty in the lower bound calculation.
In essence, these bounds are obtained by conditioning on
the values of the processes $k \mapsto X_{r,k}(h_i)$, close to the ``branching point'' $h_1 \wedge h_2$ (defined in \eqref{eqn: branching}), 
and then applying the following two lemmas.
Lemma \ref{lem:Coupled Bound} gives an estimate
for the part of the event before the branching point (where the processes are coupled),
and Lemma \ref{lem: Decoupled bound} for the part after (where they are decoupled).
To get sufficiently strong coupling and decoupling, each estimate must be applied for scales that are respectively slightly before and slightly after the branching point. 
To quantify this, we use for the decoupling parameter $\Delta$ the value
\begin{equation}\label{eq: Delta def}
\Delta = r/100\ .
\end{equation}
For convenience, define the recentered process
\[
\overline{X}_{r,k}\left(h\right)=X_{r,k}\left(h\right)-\left(k-r\right)\mu(-\varepsilon) \ .
\]
\begin{lem}
\label{lem:Coupled Bound}
Let $h_{1},h_{2}\in\mathbb{R}$ and $l= h_1\wedge h_2$. 
For $i=1,2$ and any $q\geq 0$, define the event
\begin{equation}
\label{eqn: A(q)}
 A_i (q)= \left\{ \overline{X}_{r,l-\Delta}\left(h_{i}\right)\in\left[-q,-q+1\right],\overline{X}_{r,k}\left(h_{i}\right)\le1\mbox{ for }k=r+1,\ldots,l-\Delta\right\}.
\end{equation}
Then for any $q_{1},q_{2}\ge0$,
\begin{equation}\label{eq: coupled bound}
\mathbb{P}\left[A_1(q_1) \cap A_2(q_2) \right]\\
\le c\ \frac{e^{-\left(l-\Delta-r\right)\frac{\mu(-\varepsilon)^{2}}{2\sigma^2}}}{\left(l-\Delta-r\right)^{3/2}}\left(1+q_{1}\right)e^{\frac{1}{2}\frac{\mu(-\varepsilon)}{\sigma^{2}}\left(q_{1}+q_{2}\right)}\ .
\end{equation}
\end{lem}
\begin{proof}
Let $\lambda=\mu(-\varepsilon)/(2\sigma^2)$ and $\vec{\lambda}=\lambda (1,1)$.
We recall the definition of $\mathbb{Q}_{\vec{\lambda}}$ from \eqref{eqn: RN1 and RN2}.
The choice of $\lambda$ ensures that
$\mathbb{Q}_{\vec{\lambda}}\left[\vec{Y}_{k}(0)\right]$ is approximately $\mu(-\varepsilon)(1,1)$.
By the definition of $\mathbb{Q}_{\vec{\lambda}}$,
\begin{equation}
\label{eqn: A1}
\mathbb{P}\left[A_1(q_1) \cap A_2(q_2) \right] = \mathbb{Q}_{\vec{\lambda}}\left[ \textbf{1}_{A_1 (q_1)\cap A_2(q_2)} {\displaystyle \prod_{i=1,2}}e^{-\lambda\overline{X}_{r,l-\Delta}\left(h_{i}\right)} \right]
 \exp{\left({ \sum_{k=r+1}^{l-\Delta}\{\psi_{k}^{\left(2\right)}\left(\vec{\lambda}\right) -2\lambda\mu(-\varepsilon)}\}\right)} ,
\end{equation}
where $\textbf{1}_{A_1(q_1)\cap A_2(q_2)}$ denotes the indicator function of the event. Using Proposition \ref{prop: two point ld} as well as the covariance estimates \eqref{eqn: sigmak estimate} and \eqref{eqn: correlation estimates}, we have that 
$$
\sum_{k=r+1}^{l-\Delta}\psi_{k}^{\left(2\right)}(\vec{\lambda})
=    \lambda^{2}\sum_{k=r+1}^{l-\Delta}\left(\sigma_{k}^{2}+\rho_{k} + O(e^{-2^{k-1}}) \right)
\le  \lambda^{2}(l-\Delta-r)2\sigma^2 + O(1)
=    (l-\Delta-r)\frac{\mu(-\varepsilon)^2}{2 \sigma^2} + O(1).
$$
This proves that the second exponential in \eqref{eqn: A1} is at most $ce^{-\left(l-\Delta-r\right)\frac{\mu(-\varepsilon)^{2}}{2\sigma^2}}$.
Also on the event $A_1(q_1) \cap A_2(q_2)$, the first exponential is at most $ce^{\lambda q_1 + \lambda q_2}$. Thus
\begin{equation}\label{eq: half-way there}
\mathbb{P}[A_1(q_1)\cap A_2(q_2)]
\le
ce^{-\left(l-\Delta-r\right)\frac{\mu(-\varepsilon)^{2}}{2\sigma^2}+\frac{1}{2}\frac{\mu(-\varepsilon)}{\sigma^{2}}\left(q_{1}+q_{2}\right)}
\mathbb{Q}_{\vec{\lambda}}\left[A_1(q_1)\cap A_2(q_2) \right].
\end{equation}
It remains to bound $\mathbb{Q}_{\vec{\lambda}}\left[ A_1(q_1)\cap A_2(q_2) \right]$. In fact, we drop the condition on $h_2$ and bound $\mathbb{Q}_{\vec{\lambda}}\left[ A_1(q_1) \right]$. We expect not to lose much by this because the behaviour at $h_1$ and $h_2$ should be very similar.
The event $A_1(q_1)$ is of the right form to use Proposition \ref{prop: two point gaussian comparison before branch} with $m=r$ and $n=l-\Delta$.
After recentering of the increments by $\mu(-\varepsilon)$, we get that $\mathbb{Q}_{\vec{\lambda}}\left[A_{1}(q_1)\right]$ is
\[
(1+O(e^{-cr}))
\eta_{0,\sigma^{2}}^{\times(l-\Delta-r)}\left\{ \vec{y}\in\mathbb{R}^{l-\Delta-r}:\begin{array}{c}
\sum_{l'=1}^{k}y_{l'}\le 1 \mbox{ for }k=1,\ldots,l-\Delta -r,\\
\sum_{l'=1}^{l-\Delta-r}y_{l'}\in\left[-q_1,-q_1+1\right]
\end{array}\right\} +O(e^{-e^{cr}})\ .
\]
By \eqref{eq:ballot theorem ub} of the ballot theorem (Lemma \ref{lem:ballot theorem})
with $b = -q_1$ and $\delta=1$ the probability on the right-hand side
is at most $c\frac{1+q_{1}}{(l-\Delta-r)^{3/2}}$. Together with \eqref{eq: half-way there}
this proves \eqref{eq: coupled bound}.
\end{proof}
We now prove the bound for scales after the decoupling point.
One notable difference with the proof of the previous lemma is that the change of measure is now done for a $\lambda$ which is twice the one of Lemma \ref{lem:Coupled Bound}. 
This reflects the fact that, before the branching point, the two processes are essentially coupled, therefore a tilt for one process is also a tilt for the other.
\begin{lem}
\label{lem: Decoupled bound}
Let $h_{1},h_{2}\in\mathbb{R}$. 
For any $h_{1}\wedge h_{2}+\Delta\leq j \leq n$,
and $\delta,\delta'>0$, define for $i=1,2$  and $q\geq 0$ the events
\begin{equation}
\label{eqn: events B}
\begin{aligned}
  B_i(q)&= \left\{ \overline{X}_{j,n}\left(h_{i}\right)-q\in\left[-\delta',\delta\right],\overline{X}_{j,k}\left(h_{i}\right)-q\le 1\mbox{ for }k=j+1,\ldots,n\right\}\\
  \overline{B}_{i} (q)&=
\Big\{\vec{y}\in\left(\mathbb{R}^{2}\right)^{\times\left(n-j\right)}:
  \sum_{k=1}^{n-j}\left(\vec{y}_{k}\right)_{i}-q\in\left[-\delta',\delta\right],
  \sum_{k=1}^{j'}\left(\vec{y}_{k}\right)_{i}-q\le 1,\forall j'=1,\ldots,n-j
\Big\},
\end{aligned}
\end{equation}
where $\vec{y}=\big(\big(\left(\vec{y}_{k}\right)_{1},\left(\vec{y}_{k}\right)_{2}\big), k=1,\dots,n-j\big)$. Then  for $q_{1},q_{2}\in\mathbb{R}$,
\begin{equation}\label{eq: events B bound}
  \mathbb{P}\left[B_1(q_1) \cap B_2(q_2) \right] \le \left(1+o\left(1\right)\right)e^{c\delta'}\prod_{i=1,2}\left\{ e^{-\left(n-j \right)\frac{\mu(-\varepsilon)^{2}}{2\sigma^2}-\frac{\mu(-\varepsilon) }{\sigma^{2}}q_{i}} ( \eta_{0,\sigma^{2}}^{\times\left(n-j \right)}\left[\overline{B}_i(q_i)\right] +  e^{-e^{c\Delta}} ) \right\} \ .
\end{equation}
\end{lem}

\begin{proof}
Let $\lambda=\frac{\mu(-\varepsilon)}{\sigma^{2}}$, $\vec{\lambda}=\lambda \left(1,1\right)$
and recall the definition of $\mathbb{Q}_{\vec{\lambda}}$ from \eqref{eqn: RN1 and RN2}.
The choice of $\vec{\lambda}$ ensures that $\mathbb{Q}_{\vec{\lambda}}\left[\vec{Y}_{k}\right]$ is approximately $\mu(-\varepsilon) \left(1,1\right)$.
The definition of $\mathbb{Q}_{\vec{\lambda}}$ gives
\begin{equation}
\label{eqn: low scales prob}
\mathbb{P}\left[B_1(q_1) \cap B_2(q_2) \right]
=
\mathbb{Q}_{\vec{\lambda}}\left[ \textbf{1}_{B_1(q_1) \cap B_2(q_2)}  {\displaystyle \prod_{i=1,2}}e^{-\lambda\overline{X}_{j,n}\left(h_{i}\right)} \right]
e^{\sum_{k=j+1}^{n}\psi{}_{k}^{\left(2\right)}\left(\vec{\lambda}\right)-2\lambda (n-j) \mu(-\varepsilon) }.
\end{equation}
By Proposition \ref{prop: two point ld}, \eqref{eqn: sigmak estimate} and \eqref{eqn: correlation estimates},
$$\psi_{k}^{\left(2\right)}(\vec{\lambda})=\lambda^2(\sigma_{k}^2 + \rho_k)
+ O(e^{-2^{k-1}})
=\lambda \mu(-\varepsilon) + O\left(2^{-\left(k- h\wedge h_2 \right)}\right)\ .
$$
We deduce that $\sum_{k=j+1}^{n}\psi_{k}^{\left(2\right)}(\vec{\lambda})$ is at most $\left(n-j\right)\lambda \mu(-\varepsilon)+c2^{-\Delta}$.
Therefore, the second exponential in \eqref{eqn: low scales prob} is $(1+o(1))e^{-2\left(n-j\right)\frac{\mu(-\varepsilon)^{2}}{2\sigma^2}}$.
On the event $B_1(q_1) \cap B_2(q_2)$, the first exponential in \eqref{eqn: low scales prob} is at most $e^{c\delta'}e^{ -\frac{\mu(-\varepsilon)}{\sigma^2}q_1 - \frac{\mu(-\varepsilon)}{\sigma^2}q_2}$.
In view of this, it only remains to show
\begin{equation}
\label{eq: two point independent to show}
\mathbb{Q}_{\vec{\lambda}}\left[B_1(q_1) \cap B_2(q_2) \right]
\le
\left(1+o\left(1\right)\right)\prod_{i=1,2}\eta_{0,\sigma^2}^{\times\left(n-j\right)}\left[\overline{B}_i (q_i)\right] + ce^{-e^{c\Delta}}.
\end{equation}
Note that the event $B_1(q_1) \cap B_2(q_2)$ takes the form considered in Proposition \ref{prop: two point gaussian comparison}. 
Applying Proposition \ref{prop: two point gaussian comparison} with $j$ in place of $m$ and then recentering yields
$$
\mathbb{Q}_{\vec{\lambda}}\left[B_1(q_1) \cap B_2(q_2) \right]\le\left(1+ce^{-c\Delta}\right)\eta_{0,\sigma^{2}}^{\times 2\left(n-j\right)}\left[\overline{B}_{1}(q_1)\cap\overline{B}_{2}(q_2)\right]+ce^{-e^{c\Delta}}\ .
$$
By independence, it is plain that
$\eta_{0,\sigma^{2}}^{\times 2\left(n-j\right)}\left[\overline{B}_{1}(q_1)\cap\overline{B}_{2}(q_2)\right]
=
\prod_{i=1,2}\eta_{0,\sigma^{2}}^{\times\left(n-j\right)}\left[\overline{B}_{i}(q_i)\right]$. 
This proves \eqref{eq: two point independent to show} and therefore also \eqref{eq: events B bound}.
\end{proof}
The previous lemmas will now be used to prove bounds on
$\mathbb{P}\left[J^{-}\left(h_{1}\right)\cap J^{-}\left(h_{2}\right)\right]$
in three cases: i) $h_1\wedge h_2\leq r-\Delta$, 
ii) $r-\Delta<h_1\wedge h_2\leq r+\Delta$, and iii) $r+\Delta<h_1\wedge h_2\leq n-\Delta$.
The case $h_1\wedge h_2>n-\Delta$ is easy and will be handled directly in the proof of Lemma \ref{lem: second moment}.

If $h_1\wedge h_2\leq r-\Delta$ then $h_{1}$ and $h_{2}$ are sufficiently far apart so that the scale 
$r$ is well beyond the ``branching point'' of $h_{1}$ and $h_{2}$, and the events $J^{-}\left(h_{1}\right)$ and $J^{-}\left(h_{2}\right)$ decouple:
\begin{lem}
\label{lem: low scales}
Let $h_{1},h_{2}\in\mathbb{R}$ be such that
$1\le h_{1}\wedge h_{2}\le r-\Delta$. Then
\begin{equation}
\label{eq: two point bound independent}
\mathbb{P}\left[J^{-}\left(h_{1}\right)\cap J^{-}\left(h_{2}\right)\right]\le\left(1+o\left(1\right)\right)\left(\frac{\left(n-r\right)^{\frac{3}{2}+2\varepsilon}}{2^{n-r}}\eta_{0,\sigma^{2}}^{\times\left(n-r\right)}\left[A\right]\right)^{2},
\end{equation}
where $A$ is the event defined in \eqref{eqn: A}.
\end{lem}
\begin{proof}
Let $j=r$. By assumption we have $h_{1}\wedge h_{2}+\Delta\le j$,
so Lemma \ref{lem: Decoupled bound} can be applied with $q_{1}=q_2=0$ and $\delta'=0$ to give,
\[
\mathbb{P}\left[J^{-}\left(h_{1}\right)\cap J^{-}\left(h_{2}\right)\right]\le\left(1+o\left(1\right)\right)\left(e^{-\frac{\mu(-\varepsilon)^{2}}{2 \sigma^2}\left(n-r\right)}\left(\eta_{0,\sigma^{2}}^{\times\left(n-r\right)}\left[A\right] + e^{-e^{c\Delta}}\right)\right)^{2}.
\]
By \eqref{eq: A lb from ballot} and \eqref{eq: Delta def} the probability
%The ballot theorem (Lemma \ref{lem:ballot theorem}) implies that
$\eta_{0,\sigma^{2}}^{\times\left(n-r\right)}\left[A\right]$ dominates $e^{-e^{c\Delta}}$, so the claim follows by \eqref{eq: mu squared}.
\end{proof}
In the case where $h_{1}$ and $h_{2}$ are such that their ``branching point'' happens after the scale $r+\Delta$, there is no hope of a decoupling of $J^{-}\left(h_{1}\right)$ and $J^{-}\left(h_{2}\right)$. 
Instead, we need to split the probability into a coupled part and a decoupled part and use Lemmas \ref{lem:Coupled Bound} and \ref{lem: Decoupled bound} separately.
\begin{lem}
\label{lem: branching after r+Delta}
Let $h_1,h_2\in\mathbb{R}$ and $l= h_1 \wedge h_2$.
If $r+\Delta< l \leq n-\Delta$, then
\begin{equation}
\label{eq: two point bound branching after r}
\mathbb{P}\left[J^{-}\left(h_{1}\right)\cap J^{-}\left(h_{2}\right)\right]\le 
c\  2^{-(2n-l)}\ 2^{19\Delta+r}  \ 
\frac{(n-r)^{\left(\frac{3}{2}+2\varepsilon\right)\left(2-\frac{l+3\Delta-r}{n-r}\right)}}{\left(n-l-\Delta\right)^{3}(l-\Delta-r)^{3/2}}\ .
\end{equation}
\end{lem}

\begin{proof}
Write $\overline{X}_{r,n}\left(h\right)=\overline{X}_{r,l-\Delta}\left(h\right)+\overline{X}_{l-\Delta, n}\left(h\right)$
and decompose the event $J^{-}\left(h_{1}\right)\cap J^{-}\left(h_{2}\right)$ over the values of $\overline{X}_{r,l-\Delta}\left(h\right)$ as follows
\[
\bigcup_{q_{1},q_{2}=0}^{\infty}\left( J^{-}\left(h_{1}\right)\cap J^{-}\left(h_{2}\right)\cap\bigcap_{i=1}^{2}\left\{ \overline{X}_{r,l-\Delta}\left(h_{i}\right)\in\left[-q_{i},-q_{i}+1\right]\right\} \right) \ .
\]
For fixed $q_{1},q_{2}$, the event in the union is contained in
$
\bigcap_{i=1,2}A_{i}\left(q_{i}\right)\cap C_{i}\left(q_{i}\right),
$
where the events $A_i(q)$ are defined in \eqref{eqn: A(q)} and for $i=1,2$,
\[
\begin{array}{l}
C_{i}\left(q\right)=\left\{ \overline{X}_{l-\Delta,n}\left(h_{i}\right)\in\left[q-1,q+\delta\right],\overline{X}_{l-\Delta,k}\left(h_{i}\right)\le q+1\mbox{ for }k=l-\Delta,\ldots,n\right\} .
\end{array}
\]
Now note that $ \left( \overline{X}_{r,k}\left(h_{i}\right) \right)_{r\le k\le l - \Delta},i=1,2,$ are
independent from $\left( \overline{X}_{l-\Delta,k}\left(h_{i}\right) \right)_{l - \Delta \le k \le n},i=1,2.$
Altogether we get that
\begin{equation}
\label{eqn: ABC1}
\mathbb{P}\left[J^{-}\left(h_{1}\right)\cap J^{-}\left(h_{2}\right)\right]\le\sum_{q_{1},q_{2}=0}^{\infty}\mathbb{P}\left[A_1(q_1) \cap A_2(q_2)\right]\mathbb{P}\left[ C_1(q_1) \cap C_2(q_2) \right].
\end{equation}
Lemma \ref{lem:Coupled Bound} gives
$$
\mathbb{P}\left[A_1(q_1) \cap A_2(q_2)\right]
\le
c\frac{e^{-\left(l-\Delta-r\right) \frac{\mu(-\varepsilon)^{2}}{2\sigma^2}}}{\left(l-\Delta-r\right)^{3/2}} \left(1+q_{1}\right) e^{\frac{\mu(-\varepsilon)}{2\sigma^{2}}\left(q_{1}+q_{2}\right)}.
$$
In order to use Lemma \ref{lem: Decoupled bound}, we express the probability on the event $C_i$'s by conditioning on $\overline{X}_{l-\Delta,l+\Delta}(h_{i})$, which are independent of $\overline{X}_{l+\Delta,n}(h_{i})$. We have 
\begin{equation}
\label{eqn: ABC3}
\mathbb{P}\left[ C_1(q_1) \cap C_2(q_2) \right]
=\int_{\R^2}\prob\left[B_1(q_1-y_1)\cap B_1(q_2-y_2)\right] \ f(y_1,y_2) dy_1dy_2
\end{equation}
where $f(y_1,y_2)$ is the density of $(\overline{X}_{l-\Delta,l+\Delta}(h_i), i=1,2)$
, and the events $B_i$'s are as in \eqref{eqn: events B} with $\delta'=1$. 
Lemma  \ref{lem: Decoupled bound} then gives
\begin{equation}
\label{eqn: ABC4}
\mathbb{P}\left[ B_1(q_1-y_1) \cap B_2(q_2-y_2) \right]
\le
c
\frac{e^{-2\left(n-l-\Delta\right) \frac{\mu(-\varepsilon)^{2}}{2\sigma^2}}}{\left(n-l-\Delta\right)^3}
\prod_{i=1,2}(1+q_i-y_i)e^{-\frac{\mu(-\varepsilon)}{\sigma^{2}}\left(q_{i}-y_i\right)},
\end{equation}
using also that $\eta_{0,\sigma^{2}}^{\times\left(n-l\right)}\left[\overline{B}_i(q_i-y_i)\right] \le c(1+q_i-y_i)/\left(n-l-\Delta\right)^{3/2}$
by \eqref{eq:ballot theorem ub} of the ballot theorem with $\delta' + \delta$  in place of $\delta$ and $b = q_i - y_i - \delta'$. Thus
$$
\mathbb{P}\left[J^{-}\left(h_{1}\right)\cap J^{-}\left(h_{2}\right)\right]
\le 
\int_{\R^2} \prod_{i=1,2}(1+q_i-y_i)e^{-\frac{\mu(-\varepsilon)}{\sigma^{2}}\left(q_i - y_i\right)}f(y_1,y_2) dy_1dy_2.
$$
To handle the integral, note that Proposition \ref{prop: two point ld} implies
\begin{equation}
\label{eqn: ABC5}
\E[e^{\frac{\mu(-\varepsilon)}{\sigma^{2}}(\sum_{i=1,2}\overline{X}_{l-\Delta,l+\Delta}(h_{i}))}]
\leq c \exp\left(\sum_{k=l-\Delta+1}^{l+\Delta}\frac{\mu(-\varepsilon)^2}{\sigma^{4}}(\sigma_k^2+\rho_k)\right)
\leq c e^{ \Delta \ 16\log 2}\ .
\end{equation}
where the last inequality follows from \eqref{eqn: sigmak estimate} and the inequalities $\rho_k\leq \sigma_k^2 \leq 2 \sigma^2$ and
$\mu(-\varepsilon)/\sigma^{2} \le 2$ (see \eqref{eq: mu squared}).
Using \eqref{eq: two point mean and variance}, the same estimate holds for $\E[\overline{X}_{l-\Delta,l+\Delta}(h_{1})e^{\frac{\mu(-\varepsilon)}{\sigma^{2}}(\sum_{i=1,2}\overline{X}_{l-\Delta,l+\Delta}(h_{i}))}]$
and $\E[\prod_{i=1,2}\overline{X}_{l-\Delta,l+\Delta}(h_{i})e^{\frac{\mu(-\varepsilon)}{\sigma^{2}}\overline{X}_{l-\Delta,l+\Delta}(h_{i})}]$.
Altogether this implies
\begin{equation}
\label{eqn: ABC6}
\int_{\R^2} \prod_{i=1,2}(1+q_i-y_i)e^{\frac{\mu(-\varepsilon)}{\sigma^{2}}\left(y_i\right)} f(y_1,y_2) dy_1dy_2
\leq c  (1+q_1)(1+q_2) e^{ \Delta \ 16\log 2}\ .
\end{equation}
Thus, equations \eqref{eqn: ABC1} to \eqref{eqn: ABC6} yield
\[
\mathbb{P}\left[J^{-}\left(h_{1}\right)\cap J^{-}\left(h_{2}\right)\right]
\le
c 2^{ 16 \Delta}
\frac{
e^{- \left( 2(n-l-\Delta) +(l-\Delta-r) \right) \frac{\mu(-\varepsilon)^{2}}{2\sigma^2}}
}{
  \left(l-\Delta-r\right)^{3/2}\left(n-l-\Delta\right)^{3}}\ ,
\]
where we used the fact that   $\sum_{q_{1},q_{2}=0}^{\infty}
  \left(1+q_{1}\right)^2\left(1+q_2\right)e^{-cq_1-cq_2 }$ is finite. The claim then follows from \eqref{eq: mu squared}.
  \end{proof}
The case where the branching point is between $r-\Delta$ and $r+\Delta$ is handled similarly.
\begin{lem}
\label{lem: buffer}
Let $h_{1},h_{2}\in\mathbb{R}$ be such that
$r-\Delta\le h_{1}\wedge h_{2}\le r+\Delta$. Then
\begin{equation}
\label{eq: buffer}
\mathbb{P}\left[J^{-}\left(h_{1}\right)\cap J^{-}\left(h_{2}\right)\right]\le
c 2^{18 \Delta}2^{-2(n-l-\Delta)} (n-r)^{4\e}\ ,
\end{equation}
where $l :=h_1\wedge h_2$.
\end{lem}
\begin{proof}
Since $r-\Delta<l\leq r+\Delta$, we have the decomposition 
$\overline{X}_{r,n}(h)=\overline{X}_{r,l+\Delta}(h)+\overline{X}_{l+\Delta,n}(h)$. 
We proceed as in Lemma \ref{lem: branching after r+Delta} by conditioning on $\overline{X}_{r,l+\Delta}(h_i)$, $i=1,2$,
and then drop the barrier condition on $\overline{X}_{r,l+\Delta}(h_i)$ for both $i=1$ and $i=2$. 
Following \eqref{eqn: ABC3} and \eqref{eqn: ABC4}, this gives
$$
\mathbb{P}\left[J^{-}\left(h_{1}\right)\cap J^{-}\left(h_{2}\right)\right]\le
c\frac{e^{-2\left(n-l-\Delta\right) \frac{\mu(-\varepsilon)^{2}}{2\sigma^2}}}{\left(n-l-\Delta\right)^3}
\int_{\R^2}\prod_{i=1,2}(1-y_i) e^{\frac{\mu(-\varepsilon)}{\sigma^{2}}y_i} f(y_1,y_2)dy_1dy_2\ ,
$$
where $f(y_1,y_2)$ is now the density of $(\overline{X}_{r, l+\Delta}(h_i), i=1,2)$.
The integral can be estimated using Proposition \ref{prop: two point ld} as in \eqref{eqn: ABC5}. 
It is smaller than $c2^{16\Delta}$. By \eqref{eq: mu squared}, the fraction in front of the integral is
$$
2^{-2(n-l-\Delta)} (n-r)^{4\e \frac{n-l-\Delta}{n-r}}\ (n-r)^{3\frac{n-l-\Delta}{n-r}}/\left(n-l-\Delta\right)^3 \ .
$$
Since $r-\Delta<l<r+\Delta$, this is smaller than $c 2^{-2(n-l-\Delta)}(n-r)^{4\e}$ as claimed. 
\end{proof}

We now have the necessary two-point estimates to prove the upper bound on $\E[(Z^{-})^2]$.

\begin{proof}[Proof of Lemma \ref{lem: second moment}]
We split the sum in \eqref{eqn: second moment two-point} into four terms depending on the branching point $h_1\wedge h_2$ of the pair $h_1,h_2\in \mathcal{H}_{n}\cap[0,1)$:
$$
\underset{\left(I\right)}{\underbrace{\sum_{h_{1},h_{2}: \ h_1\wedge h_2\leq r-\Delta}(\cdot)}} +
\underset{\left(II\right)}{\underbrace{\sum_{h_{1},h_{2}:\ r-\Delta< h_1\wedge h_2\leq r+\Delta}(\cdot)}} 
+\underset{(III)}{\underbrace{\sum_{h_{1},h_{2}:\ r+\Delta<h_{1}\wedge h_{2} < n-\Delta}(\cdot)}}
+\underset{\left(IV\right)}{\underbrace{\sum_{h_{1},h_{2}: \ h_1\wedge h_2 \ge n-\Delta}(\cdot)}}\ .
$$
Using that $\#\mathcal{H}_{n}\cap[0,1)=2^{n}$ and the bound \eqref{eq: two point bound independent}, we get
$$
\left(I\right)\le (1+o(1)) \left(2^r (n-r)^{\frac{3}{2}+2\varepsilon}\eta_{0,\sigma^2}^{\times\left(n-r\right)}\left[A\right]\right)^{2}\ .
$$
By \eqref{eq: A lb from ballot} the right-hand side is at least $c2^{2r} (n-r)^{4\varepsilon}$.
We now show that $(II)$, $(III)$ and $(IV)$ are negligible compared to this,
and thus (I) is the dominant term in the sum.
Note that the number of pairs $h_1,h_2 \in \mathcal H_n \cap [0,1)$ such that $2^{-l-1} \le |h_1-h_2| \le 2^{-l}$ is at most $c2^{2n-l}$.
Thus the contribution of $(II)$, by Lemma
\ref{lem: buffer}, is at most
$$
\left(II\right)\leq c \sum_{l=r-\Delta+1}^{r+\Delta}2^{2n-l}\ 2^{16 \Delta}2^{-2(n-l-\Delta)} (n-r)^{4\e}
\leq c 2^{19\Delta} 2^r (n-r)^{4\e}\ ,
$$
which is negligible compared to $2^{2r} (n-r)^{4\e}$, because of the choice $\Delta=r/100$.
Similarly, the contribution of $(III)$ can be bounded as
\[
\left(III\right)\leq \sum_{l=r+\Delta+1}^{n-\Delta -1}2^{2n-l}\max_{ h \in [2^{-l-1}, 2^{l}] } \mathbb{P}\left[J^{-}\left(0\right)\cap J^{-}\left(h\right)\right]\ .
\]
Lemma \ref{lem: branching after r+Delta} then yields
$$
\begin{aligned}
\left(III\right)
&\le c 2^{r+19 \Delta}\left(n-r\right)^{4\varepsilon}
\sum_{l=r+\Delta+1}^{n-\Delta-1} \frac{(n-r)^{\frac{3}{2}\left(2-\frac{l+3\Delta-r}{n-r}\right)}}{(n-l-\Delta)^3(l-\Delta-r)^{3/2}}\\
&=c 2^{r+19 \Delta}\left(n-r\right)^{4\varepsilon}\sum_{a=1}^{m-2\Delta-1}\frac{m^{\frac{3}{2}\left(2-(a+2\Delta)/m\right)}}{(m-a - 2\Delta)^{3}a^{3/2}}\ , \text{ for $m=n-r$}\\
&\le c2^{r+19 \Delta}\left(n-r\right)^{4\varepsilon},
\end{aligned}
$$
where the last inequality follows from the fact that the sum over $a$ stays finite as $m\to\infty$.
Since $\Delta=r/100$, the bound on $(III)$ is negligible relative to the bound on $(I)$.
Finally, for $(IV)$, the event $J^{-}\left(h_{2}\right)$ can be dropped.
There are at most $2^{n+\Delta}$ pairs $h_{1},h_{2}\in\mathcal{H}_{n} \cap [0,1)$ such that $\left|h_{1}-h_{2}\right|\le2^{-n+\Delta}$.
 A union bound using the one-point bound \eqref{eqn: 1 point rate upper}  gives
$$
\left(IV\right) \leq 2^{n+\Delta}\prob[J^{-}(0)]\leq (1+o(1)) \ 2^{r+\Delta} (n-r)^{2\varepsilon}\ .
$$
 Again, this is negligible relative to the bound on $(I)$. Therefore
 $$
 \left(I\right)+\left(II\right)+\left(III\right)+\left(IV\right)\leq  (1+o(1)) \left(2^r (n-r)^{\frac{3}{2}+2\varepsilon}\eta_{0,\sigma^2}^{\times\left(n-r\right)}\left[A\right]\right)^{2}\ ,
 $$
 which proves the lemma. 
\end{proof}
This bound on the second moment of $Z^{-}$
concludes the proof of lower bound Proposition \ref{prop: lower bound}, and therefore also
for the main result Theorem \ref{thm: main}. 

\bibliographystyle{plain}

\bibliography{bib_branching_primes}

\begin{thebibliography}{10}

\bibitem{addario-berry-reed-minima}
L.~Addario-Berry and B.~Reed.
\newblock Minima in branching random walks.
\newblock {\em Ann. Probab.}, 37(3):1044--1079, 2009.

\bibitem{aidekon}
E.~A{\"{\i}}d{\'e}kon.
\newblock Convergence in law of the minimum of a branching random walk.
\newblock {\em Ann. Probab.}, 41(3A):1362--1426, 2013.

\bibitem{aidekon-berestycki-brunet-shi}
E.~A{\"{\i}}d{\'e}kon, J.~Berestycki, {E}. Brunet, and Z.~Shi.
\newblock Branching {B}rownian motion seen from its tip.
\newblock {\em Probab. Theory Related Fields}, 157(1-2):405--451, 2013.

\bibitem{arguin-bovier-kistler}
L.-P. Arguin, A.~Bovier, and N.~Kistler.
\newblock The extremal process of branching {B}rownian motion.
\newblock {\em Probab. Theory Related Fields}, 157(3-4):535--574, 2013.

\bibitem{belius-kistler}
D.~Belius and N.~Kistler.
\newblock The subleading order of two dimensional cover times.
\newblock {\em Preprint}, arXiv:1405.0888, 2014.

\bibitem{bhattacharya-rao}
R.~N. Bhattacharya and R.~Ranga~Rao.
\newblock {\em Normal approximation and asymptotic expansions}.
\newblock John Wiley \& Sons, New York-London-Sydney, 1976.
\newblock Wiley Series in Probability and Mathematical Statistics.

\bibitem{biskup-louidor}
M.~Biskup and O.~Louidor.
\newblock Extreme local extrema of two-dimensional discrete gaussian free
  field.
\newblock {\em Preprint}, arXiv:1306.2602, 2013.

\bibitem{bourgade}
P.~Bourgade.
\newblock Mesoscopic fluctuations of the zeta zeros.
\newblock {\em Probab. Theory Related Fields}, 148(3-4):479--500, 2010.

\bibitem{bramson}
M.~Bramson.
\newblock Maximal displacement of branching {B}rownian motion.
\newblock {\em Comm. Pure Appl. Math.}, 31(5):531--581, 1978.

\bibitem{bramson-ding-zeitouni}
M.~Bramson, J.~Ding, and O.~Zeitouni.
\newblock Convergence in law of the maximum of the two-dimensional discrete
  gaussian free field.
\newblock {\em Preprint}, arXiv:1301.6669, 2013.

\bibitem{bramson-ding-zeitouni2}
M.~Bramson, J.~Ding, and O.~Zeitouni.
\newblock Convergence in law of the maximum of nonlattice branching random
  walk.
\newblock {\em Preprint}, arxiv: 1404.3423, 2014.

\bibitem{farmer-gonek-hughes}
D.~W. Farmer, S.~M. Gonek, and C.~P. Hughes.
\newblock The maximum size of {$L$}-functions.
\newblock {\em J. Reine Angew. Math.}, 609:215--236, 2007.

\bibitem{fyodorov-bouchaud}
Y.~V. Fyodorov and J.-P. Bouchaud.
\newblock Freezing and extreme-value statistics in a random energy model with
  logarithmically correlated potential.
\newblock {\em J. Phys. A}, 41(37):372001, 12, 2008.

\bibitem{fyodorov-hiary-keating}
Y.~V. Fyodorov, G.~A. Hiary, and J.~P. Keating.
\newblock Freezing transition, characteristic polynomials of random matrices,
  and the {Riemann} zeta function.
\newblock {\em Phys. Rev. Lett.}, 108:170601, Apr 2012.

\bibitem{fyodorov-keating}
Y.~V. Fyodorov and J.~P. Keating.
\newblock Freezing transitions and extreme values: random matrix theory, and
  disordered landscapes.
\newblock {\em Philos. Trans. R. Soc. Lond. Ser. A Math. Phys. Eng. Sci.},
  372(2007):20120503, 32, 2014.

\bibitem{fyodorov-simms}
Y.V. Fyodorov and N.J. Simm.
\newblock On the distribution of maximum value of the characteristic polynomial
  of gue random matrices.
\newblock {\em Preprint}, arXiv:1503.07110, 2015.

\bibitem{goetze}
F.~G{\"o}tze.
\newblock On the rate of convergence in the multivariate {CLT}.
\newblock {\em Ann. Probab.}, 19(2):724--739, 1991.

\bibitem{halasz}
G.~Hal{\'a}sz.
\newblock On random multiplicative functions.
\newblock In {\em Hubert {D}elange colloquium ({O}rsay, 1982)}, volume~83 of
  {\em Publ. Math. Orsay}, pages 74--96. Univ. Paris XI, Orsay, 1983.

\bibitem{harper}
A.~J. Harper.
\newblock A note on the maximum of the {R}iemann zeta function, and
  log-correlated random variables.
\newblock {\em Preprint}, arxiv: 1304.0677, 2013.

\bibitem{keating-snaith}
J.~P. Keating and N.~C. Snaith.
\newblock Random matrix theory and {$\zeta(1/2+it)$}.
\newblock {\em Comm. Math. Phys.}, 214(1):57--89, 2000.

\bibitem{kistler}
N.~Kistler.
\newblock Derrida's random energy models. {F}rom spin glasses to the extremes
  of correlated random fields.
\newblock In {\em Correlated random systems: five different methods}, volume
  2143 of {\em Lecture Notes in Math.}, pages 71--120. Springer, Cham, 2015.

\bibitem{madaule}
T.~Madaule.
\newblock Maximum of a log-correlated gaussian field.
\newblock {\em Preprint}, arXiv:1307.1365, 2013.

\bibitem{montgomery-vaughan-multiplicative-nt}
H.~L. Montgomery and R.~C. Vaughan.
\newblock {\em Multiplicative number theory. {I}. {C}lassical theory},
  volume~97 of {\em Cambridge Studies in Advanced Mathematics}.
\newblock Cambridge University Press, Cambridge, 2007.

\bibitem{selberg}
A.~Selberg.
\newblock Contributions to the theory of the {R}iemann zeta-function.
\newblock {\em Archiv Math. Naturvid.}, 48(5):89--155, 1946.

\bibitem{soundextremes}
K.~Soundararajan.
\newblock Extreme values of zeta and ${L}$-functions.
\newblock {\em Math. Ann.}, 342(2):467--486, 2008.

\bibitem{sound}
K.~Soundararajan.
\newblock Moments of the {R}iemann zeta function.
\newblock {\em Ann. of Math. (2)}, 170(2):981--993, 2009.

\bibitem{WebbbExactAsymptoticsofFreezingTransitionOfLogCorrelatedREM}
C.~Webb.
\newblock Exact asymptotics of the freezing transition of a logarithmically
  correlated random energy model.
\newblock {\em Journal of Statistical Physics}, 145(6):1595--1619, 2011.

\bibitem{webb}
C.~Webb.
\newblock The characteristic polynomial of a random unitary matrix and
  {Gaussian} multiplicative chaos - the $l^2$-phase.
\newblock {\em Preprint}, arxiv: 1410.0939, 2014.

\end{thebibliography}

\end{document}